\documentclass[12pt, a4paper, abstracton]{scrartcl}
\pdfoutput=1

\usepackage{amsmath, amsthm, amsfonts, amssymb, comment}
\usepackage[utf8]{inputenc}
\usepackage[nottoc]{tocbibind} % Bibliography im toc
\usepackage{color}
\usepackage{slashed} % F\"{u}r die durchgestrichenen Dirac-Operator-Zeichen.
\usepackage[all, cmtip]{xy}
\usepackage{hyperref} % Anklickbare Links und Inhaltsverzeichnis. Und bookmarks im pdf.
\usepackage{chngcntr} % Um die Fusnotennummerierung fortlaufend im ganzen Dokument machen zu koennen.
\usepackage{graphicx} % Damit du groessere \cdot definieren kannst. Und zum Bilder einbinden.
\usepackage[english]{babel}
\usepackage{microtype}
\usepackage{bm} % For bold + italic in mathmode
\usepackage{mathtools} % \mathclap
\usepackage{mathabx}

\newcommand{\IN}{\mathbb{N}}
\newcommand{\IZ}{\mathbb{Z}}
\newcommand{\IQ}{\mathbb{Q}}
\newcommand{\IR}{\mathbb{R}}
\newcommand{\IC}{\mathbb{C}}

\newcommand{\calS}{\mathcal{S}}

\newcommand{\IK}{\mathfrak{K}}

\newcommand{\GL}{\operatorname{GL}}
\newcommand{\id}{\operatorname{id}}

\newcommand{\image}{\operatorname{im}}

\newcommand{\ab}{\mathrm{ab}}

\newcommand{\trace}{\operatorname{tr}}
\newcommand{\ch}{\operatorname{ch}}

\newcommand{\HC}{H \! C}
\newcommand{\HP}{H \! P}
\newcommand{\HH}{H \! H}
\newcommand{\PHC}{P \! H \! C}

\newcommand{\Poincare}{Poincar\'{e} }

\newcommand{\alg}{\mathrm{alg}}
\newcommand{\toprm}{\mathrm{top}}

\newcommand{\EG}{\underline{EG}}

\newcommand{\fin}{\mathrm{fin}}
\newcommand{\hd}{\mathrm{hd}}
\newcommand{\cd}{\mathrm{cd}}
\newcommand{\gd}{\mathrm{gd}}
\newcommand{\CAT}{\mathrm{CAT}}
\newcommand{\chSG}{\operatorname{ch}^{\calS G}}
\newcommand{\asdim}{\operatorname{as-dim}}
\newcommand{\gdunder}{\underline{\smash{\operatorname{gd}}}}
\newcommand{\EGunder}{\underline{EG}}

\newcommand{\KH}{K \! H}

\def\colim{\mathop{\mathrm{colim}}\nolimits}

\newcommand{\Michal}[1]{{\textcolor{blue}{#1}}}

\newtheorem{thm}{Theorem}[section]
\newtheorem*{thm*}{Theorem}
\newtheorem*{thmA*}{Theorem A}
\newtheorem*{thmB*}{Theorem B}
\newtheorem*{mainthm*}{Main Theorem}

\newtheorem*{roesthm*}{Roe's Index Theorem}

\newtheorem*{atiyahsingerthm*}{Atiyah--Singer Index Theorem}

\newtheorem{cor}[thm]{Corollary}
\newtheorem*{cor*}{Corollary}
\newtheorem{lem}[thm]{Lemma}
\newtheorem{prop}[thm]{Proposition}
\newtheorem{conj}[thm]{Conjecture}
\newtheorem*{conj*}{Conjecture}
\newtheorem*{conjA*}{Conjecture A}
\newtheorem*{conjB*}{Conjecture B}
\newtheorem*{obs*}{Observation}

\theoremstyle{definition}
\newtheorem{rem}[thm]{Remark}
\newtheorem*{rem*}{Remark}

\newtheorem{defn}[thm]{Definition}

\newtheorem{nota}[thm]{Notation}
\newtheorem*{GBC*}{Burghelea Conjecture}

\numberwithin{equation}{section}

\counterwithout{footnote}{section}

\makeatletter% Damit man Footnotes ohne Nummer machen kann.
\def\blfootnote{\gdef\@thefnmark{}\@footnotetext}
\makeatother

\begin{document}

\title{Burghelea conjecture and\\ asymptotic dimension of groups}
\author{Alexander Engel\thanks{Fakult\"{a}t f\"{u}r Mathematik, Universit\"{a}t Regensburg, 93040 Regensburg, GERMANY\newline
alexander.engel@mathematik.uni-regensburg.de}
\and
Micha\l{} Marcinkowski\thanks{Fakult\"{a}t f\"{u}r Mathematik, Universit\"{a}t Regensburg, 93040 Regensburg, GERMANY\newline
\hspace*{-0.5em}Instytut Matematyczny, Uniwersytet Wroc\l{}awski, 50-384 Wroc\l{}aw, POLAND\newline
michal.marcinkowski@mathematik.uni-regensburg.de}}
\date{}
\maketitle

%\vspace*{-3.5\baselineskip}
%\begin{center}
%\footnotesize{
%\textit{
%Fakultaet f\"{u}r Mathematik\\
%Universitaet Regensburg\\
%93040 Regensburg, Germany\\
%\href{mailto:Alexander.Engel@mathematik.uni-regensburg.de}{Alexander.Engel@mathematik.uni-regensburg.de}
%}}
%\end{center}
%\vspace*{0.5\baselineskip}

\begin{abstract}
We review the Burghelea conjecture, which constitutes a full computation of the periodic cyclic homology of complex group rings, and its relation to the algebraic Baum--Connes conjecture. The Burghelea conjecture implies the Bass conjecture.

We state two conjectures about groups of finite asymptotic dimension, which together imply the Burghelea conjecture for such groups.
We prove both conjectures for many classes of groups.
%We prove all the stated conjectures for non-positively curved groups, one-relator groups, elementary amenable groups of finite Hirsch length, mapping class groups, discrete subgroups of almost connected Lie groups, and $3$-manifold groups.

%All the above mentioned classes of groups have (conjecturally) finite asymptotic dimension and therefore our proofs provide considerable evidence for the correctness of the two introduced conjectures about groups of finite asymptotic dimension.\\

It is known that the Burghelea conjecture does not hold for all groups, although no finitely presentable counter-example was known. We construct a finitely presentable (even type $F_\infty$) counter-example based on Thompson's group $F$. We construct as well a finitely generated counter-example with finite decomposition complexity.
\end{abstract}

\tableofcontents

\paragraph{Acknowledgements} The authors were supported by the SFB 1085 ``Higher Invariants'' funded by the Deutsche Forschungsgemeinschaft DFG.

The first named author was also supported by the Research Fellowship EN 1163/1-1 ``Mapping Analysis to Homology'' of the Deutsche Forschungsgemeinschaft DFG, and he acknowledges the hospitality of the Hausdorff Research Institute for Mathematics in Bonn during the workshop on the Farrell--Jones conjecture, which was part of the Junior Hausdorff Trimester Program ``Topology''.

The second named author was partially supported by Polish National Science Center (ncn) grant 2012/06/A/ST/1/00259.

The authors thank Ulrich Bunke, Guillermo Corti\~{n}as, Stefan Friedl, Holger Kammeyer, Wolfgang L\"{u}ck and Roman Sauer for helpful discussions and answering a lot of questions.

\section{Introduction}

Burghelea \cite{burghelea} computed the periodic cyclic homology of complex group rings:
\[\PHC_\ast(\IC G) \cong \bigoplus_{[g] \in \langle G \rangle^{\fin}} H_{[\ast]}(C_g;\IC) \oplus \bigoplus_{[g] \in \langle G \rangle^{\infty}} T^G_\ast(g;\IC),\]
where $\langle G \rangle^{\fin}$ and $\langle G \rangle^{\infty}$ denote the conjugacy classes in $G$ of finite order elements, resp., of infinite order elements. We have used the notation $H_{[\ast]}(C_g;\IC) := \prod_{i \in \IZ} H_{\ast+2i}(C_g;\IC)$ for the centralizer $C_g$ of the element $g \in G$, and we have
\begin{equation}
\label{eqnjk3498fe}
T^G_\ast(g;\IC) := \varprojlim \big( \cdots \xrightarrow{S} H_{\ast+2n}(C_g/\langle g\rangle;\IC) \xrightarrow{S} H_{\ast+2n-2}(C_g/\langle g\rangle;\IC) \xrightarrow{S} \cdots \big),
\end{equation}
where $S$ is the Gysin homomorphism of the fibration $B\langle g \rangle = S^1 \to B C_g \to B C_g/\langle g\rangle$.

\begin{GBC*}\label{conj:bur}
We will say that $G$ satisfies the Burghelea conjecture if the group $\bigoplus_{[g] \in \langle G \rangle^{\infty}} T^G_\ast(g;\IC)$ vanishes.
\end{GBC*}

Groups for which the Burghelea conjecture is known are, e.g., hyperbolic and arithmetic groups (Ji \cite{ji_nilpotency}) and linear groups over fields of characteristic zero and having finite rational homological dimension (Eckmann \cite{eckmann}). A thorough compilation of known results about the Burghelea conjecture will be given in Remark~\ref{jkdsf9023}.

The Burghelea conjecture is known to be false in general, i.e., there is a counter-example constructed by Burghelea \cite{burghelea}. His counter-example is not finitely generated, but recently Dranishnikov--Hull \cite{DHull} came up with a counter-example which is finitely generated. In Section~\ref{sec:counter.examples} we will construct more counter-examples with better properties, especially, we construct a finitely presented one in Section~\ref{sec:F} (it will be even of type~$F_\infty$).

Actually, Burghelea conjectured the vanishing of $\bigoplus_{[g] \in \langle G \rangle^{\infty}} T^G_\ast(g;\IC)$ only for groups admitting a finite classifying space. We took the freedom to use the term ``Burghelea conjecture'' in the more general context of all groups.

\begin{obs*}
The Burghelea conjecture implies the strong Bass conjecture.
\end{obs*}

We will discuss in Section~\ref{secj23223} how the Burghelea conjecture implies the strong Bass conjecture. Our discussion will rely on the rationalization of the algebraic Baum--Connes assembly map done in Section~\ref{secdsj239233}.

The main purpose of this paper is to prove the Burghelea conjecture for more classes of groups. Concretely, we have the following main result:

\begin{thmA*}\label{thmA}
The Burghelea conjecture is true for the following classes of groups:
\begin{itemize}
\item $\CAT(0)$ groups with finite rational homological dimension\footnote{This includes Coxeter groups, right-angeled Artin groups and virtually torsion-free $\CAT(0)$ groups.},
\item systolic groups,
\item mapping class groups,
\item discrete subgroups of almost connected Lie groups, and
\item $3$-manifold groups.
\end{itemize}
Furthermore, the Burghelea conjecture is closed under passage to relatively hyperbolic supergroups.
\end{thmA*}

All groups in the above list have (conjecturally) finite asymptotic dimension. Moreover, for other classes of groups that have finite asymptotic dimension, like hyperbolic groups or solvable groups of finite Hirsch length, the Burghelea conjecture is also known. So the conjecture is that having finite asymptotic dimension should imply the Burghelea conjecture. We will break up this single conjecture into two pieces which together imply the Burghelea conjecture for groups of finite asymptotic dimension.

\begin{conjA*}
Let $G$ be a discrete group of finite asymptotic dimension.

Then its rational homological dimension is finite, i.e., $\hd_\IQ(G) < \infty$.
\end{conjA*}

\begin{conjB*}
Let $G$ be a discrete group of finite asymptotic dimension. Denote by $C_g < G$ the centralizer of an infinite order element $g \in G$.

Then the quotient $C_g / \langle g \rangle$ also has finite asymptotic dimension.
\end{conjB*}

The second main result of this paper is to verify the above two conjectures for almost all classes of groups which are known, resp., conjectured to have finite asymptotic dimension:

\begin{thmB*}
The above two conjectures are true for
\begin{itemize}
\item all classes of groups stated in Theorem A together with the closure property for relatively hyperbolic groups,
\item one-relator groups, and
\item elementary amenable groups of finite Hirsch length.
\end{itemize}
\end{thmB*}

Note that one-relator groups and elementary amenable groups of finite Hirsch length were already known to satisfy the Burghelea conjecture, which is the reason why we did not mention them in the list in Theorem A.

If one believes that the Burghelea conjecture should be true for all groups of finite asymptotic dimension, then the next question is whether one can generalize this to a larger class of groups. A natural generalization of finite asymptotic dimension is the notion of finite decomposition complexity introduced by Guentner--Tessera--Yu \cite{fdc_1}. In Section~\ref{sec:fdc} we will construct a counter-example to the Burghelea conjecture having finite decomposition complexity.

%\Michal{The Burghelea conjecture is known to be false in general, i.e., there is a counter-example constructed by Burghelea \cite{burghelea}. 
%Originaly Burghelea itself conjectured that Generalized Bughelea conjecture (GBC) is true for groups with finite classifying space. 
%Recently A.~Dranishnikov announced a proof of the original Burghelea conjecture \cite{dran.burghelea}.
%It is natural to ask for which groups GBC holds.
%Especially interesting and natural is the class of groups with finite asymptotic dimension. 
%Up to now, this question is widely open.
%In Section \ref{sec:counter.examples} we construct more counter-examples to the Generalized Burghelea Conjecture. 
%In particular we show, that GBC does not hold for groups with finite decomposition property, which is a generalisation of
%the asymptotic dimension. 
%Note that the Bughelea's counter-example is not finitely presented, and up to now, no finitely presented counter-example was known. 
%In \ref{sec:F} we construct a finitely presented counter-example based on the Thompson's group $F$.
%} 

\section{Algebraic Baum--Connes assembly map}\label{secdsj239233}

Let $\calS = \bigcup_p \calS^p (H)$ be the Schatten class operators on some separably $\infty$-dimensional Hilbert space $H$. The $K$-theoretic Farrell--Jones assembly map for the ring $\calS$ is a map
\begin{equation*}
\label{sjkdf23r}
\mu_\ast^{\mathit{FJ}} \colon H_\ast^{\mathit{Or}G}(E_{\mathit{vcyc}} G; \mathbf{K}^{-\infty}(\calS)) \to K^{\alg}_\ast(\calS G).
\end{equation*}
Yu \cite[Section 2]{yu_algebraic_novikov} and Corti\~{n}as--Tartaglia \cite[Proof of Corollary 4.2]{ctartaglia} showed (building upon work of Corti\~{n}as--Thom \cite{CT}) that the domain $H_\ast^{\mathit{Or}G}(E_{\mathit{vcyc}} G; \mathbf{K}^{-\infty}(\calS))$ of $\mu_\ast^{\mathit{FJ}}$ identifies with the left hand side $RK^G_\ast(\underline{EG})$ of the Baum--Connes conjecture, and that $\mu_\ast^{\mathit{FJ}}$ factors the Baum--Connes assembly map $\mu_\ast^{\mathit{BC}} \colon RK^G_\ast(\underline{EG}) \to K_\ast^\toprm(C^\ast_r G)$, i.e., the following diagram commutes:
\[\xymatrix{
H_\ast^{\mathit{Or}G}(E_{\mathit{vcyc}} G; \mathbf{K}^{-\infty}(\calS)) \ar[r]^-{\mu_\ast^{\mathit{FJ}}} & K^{\alg}_\ast(\calS G) \ar[d]\\
RK^G_\ast(\underline{EG}) \ar[r]^-{\mu_\ast^{\mathit{BC}}} \ar[u]_{\cong} & K^{\toprm}_\ast(C^\ast_r G)}\]
The right vertical map is given by passing to the completion $C^\ast_r G \otimes \IK$ of $\calS G \cong \IC G \otimes \calS$.

\begin{defn}
We call
\[\mu^\alg_\ast \colon RK^G_\ast(\EG) \to K^{\alg}_\ast(\calS G)\]
the algebraic Baum--Connes assembly map.
\qed
\end{defn}

\begin{conj}[Algebraic Baum--Connes conjecture]
The algebraic assembly map $\mu^\alg_\ast$ is an isomorphism for every discrete group~$G$.
\end{conj}

We will see in Corollary~\ref{sdf23fwd} that the algebraic assembly map $\mu^\alg_\ast$ is rationally injective, which was proved by Yu \cite{yu_algebraic_novikov}. Twenty years earlier Tillmann \cite{tillmann} showed a similar result, namely rational injectivity of a map $RK_\ast(BG) \to K_\ast^{\toprm}(\IC G)$ that she constructed and which also factors the Baum--Connes assembly map.

\begin{rem}
The algebraic assembly map has an index theoretic interpretation.

Using the geometric picture for $K$-homology, we see that any element of $RK_\ast^G(\underline{EG})$ is represented by a finitely summable cycle, where the order of summability depends on the dimension of the cycle. So the analytic construction of the assembly map produces an element in $K_\ast^\alg(\calS G)$.

We do not know whether Tillmann's map $RK_\ast(BG) \to K_\ast^{\toprm}(\IC G)$ has a similar index theoretic interpretation.
\qed
\end{rem}

Let us now rationalize the algebraic assembly map, i.e., we will apply / construct Chern characters for the left and right hand sides of the algebraic assembly map to turn the algebraic Baum--Connes conjecture into a rational one, i.e., one where we do not have to worry about torsion elements. This rationalization will also provide us with the relation of the algebraic Baum--Connes conjecture to the Burghelea conjecture.

\begin{nota}
Let us write for homology groups
\[H_{[n]}(X) := \bigoplus_{i \ge 0} H_{n-2i}(X) \text{ and }H_{[\ast]}(X) := \prod_{i \in \IZ} H_{\ast+2i}(X).\]
For a group $G$ we will write $\langle G \rangle$ for its conjugacy classes, $\langle G \rangle^{\fin}$ for its conjugacy classes of finite order elements and $\langle G \rangle^{\infty}$ for its conjugacy classes of infinite order elements.
\qed
\end{nota}

\begin{thm}[{\cite[Theorem 7.3]{baum_connes_higson}}]
We have a functorial Chern character
\[\ch^G_\ast\colon RK_\ast^G(X) \to H^G_{[\ast]}(X;\IC)\]
which becomes an isomorphism after tensoring the domain with $\IC$.

Furthermore, we have for all $n \in \IN_0$
\begin{equation*}
\label{dfui39sdf}
H_n^G(\EG;\IC) \cong \bigoplus_{[g] \in \langle G \rangle^{\fin}} H_n(C_g;\IC),
\end{equation*}
where $C_g < G$ is the centralizer of the element $g$ in $G$.
\end{thm}

\begin{nota}
Let us write
\begin{align}
\label{eq:sdf23fds}
\PHC_\ast(\IC G) & := \varprojlim \big( \cdots \xrightarrow{S} \HC_{\ast+2n}(\IC G) \xrightarrow{S} \HC_{\ast+2n-2}(\IC G) \xrightarrow{S} \cdots \big)\notag\\
T^G_\ast(g;\IC) & := \varprojlim \big( \cdots \xrightarrow{S} H_{\ast+2n}(C_g/\langle g\rangle;\IC) \xrightarrow{S} H_{\ast+2n-2}(C_g/\langle g\rangle;\IC) \xrightarrow{S} \cdots \big)
\end{align}
where $S$ in the definition of $\PHC_\ast(-)$ is the periodicity operator of cyclic homology and the $S$ in \eqref{eq:sdf23fds} is the Gysin homomorphism of the fibration $B\langle g \rangle = S^1 \to B C_g \to B C_g/\langle g\rangle$.
\qed
\end{nota}

\begin{thm}[{Burghelea \cite{burghelea}; see also \cite[Eq.~(3.36) on p.~152]{khalkhali_basic}}]
We have the following isomorphisms for all $n \in \IN_0$ and $\ast = 0,1$:
\begin{alignat}{3}
& \HH_n(\IC G) && \cong \bigoplus_{[g] \in \langle G \rangle^{\fin}} H_n(C_g;\IC) && \oplus \bigoplus_{[g] \in \langle G \rangle^{\infty}} H_n(C_g;\IC)\label{eq:sdfjsdfjk43}\\
& \HC_n(\IC G) && \cong \bigoplus_{[g] \in \langle G \rangle^{\fin}} H_{[n]}(C_g;\IC) && \oplus \bigoplus_{[g] \in \langle G \rangle^{\infty}} H_n(C_g/\langle g\rangle;\IC)\notag\\
& \PHC_\ast(\IC G) && \cong \bigoplus_{[g] \in \langle G \rangle^{\fin}} H_{[\ast]}(C_g;\IC) && \oplus \bigoplus_{[g] \in \langle G \rangle^{\infty}} T^G_\ast(g;\IC)\label{eq:sdfjnk43}
\end{alignat}
\end{thm}

>From results of Corti\~{n}as--Thom \cite[Theorems 6.5.3 \& 8.2.5]{CT}) we conclude \[K_\ast^\alg(\calS G) \cong \KH_\ast(\calS G)\]
and these groups are $2$-periodic (with periodicity induced by multiplication with the Bott element). By definition $\calS G$ is the filtered colimit
\[\calS G = \operatorname{colim} \big(\cdots \hookrightarrow \calS^p G \hookrightarrow \calS^{p+1} G \hookrightarrow \cdots \big)\]
and it is known that $\KH$-theory commutes with filtered colimits. We know from the proof of \cite[Corollary 3.5]{ctartaglia} that the induced maps $\KH_\ast(\calS^p G) \to \KH_\ast(\calS^{p+1} G)$ are isomorphisms for all $p \ge 1$ and so we can conclude
\[\KH_\ast(\calS G) \cong \KH_\ast(\calS^1 G).\]
We use the Connes--Karoubi character $\KH_\ast(\calS^1 G) \to \HP^{\mathrm{cont}}_\ast(\calS^1 G)$ and write
\[\chSG_\ast\colon K_\ast^\alg(\calS G) \to \HP^{\mathrm{cont}}_\ast(\calS^1 G)\]
for the composition of all the above maps. Then we apply $\trace\colon \HP^{\mathrm{cont}}_\ast(\calS^1 G) \to \HP_\ast(\IC G)$ induced by the trace, and lastly we use the projection $\HP_\ast(\IC G) \to \PHC_\ast(\IC G)$ whose application we will suppress in our notation.

So putting everything together we get the following diagram:
\begin{equation}
\label{eq:dsf8934}
\xymatrix{RK^G_\ast(\EG) \otimes \IC\ar[d]_{\ch_\ast^G}^\cong \ar[r]^-{\mu^\alg_\ast \otimes \id_{\IC}} & K_\ast^\alg(\calS G) \otimes \IC \ar[d]_{\trace \circ \chSG_\ast}\\
H_{[\ast]}^G(\EG;\IC) \ar[d]^\cong & \PHC_\ast(\IC G) \ar[d]^\cong\\
\bigoplus_{[g] \in \langle G \rangle^{\fin}} H_{[\ast]}(C_g;\IC) \ar@{>->}[r] & \bigoplus_{[g] \in \langle G \rangle^{\fin}} H_{[\ast]}(C_g;\IC) \oplus \bigoplus_{[g] \in \langle G \rangle^{\infty}} T^G_\ast(g;\IC)}
\end{equation}

\begin{cor}[{\cite{yu_algebraic_novikov}}]\label{sdf23fwd}
The algebraic assembly map is rationally injective for every group~$G$.
\end{cor}

\begin{rem}\label{remjkds2323}
During a talk about the results of this paper Wolfgang L\"{u}ck noticed that in the above Diagram \eqref{eq:dsf8934} we may include a horizontal arrow $H_{[\ast]}^G(\EG;\IC) \to \PHC_\ast(\IC G)$, which is the Farrell--Jones assembly map for the theory $\PHC(\IC -)$. This breaks the whole diagram into two separate pieces.

Corti\~{n}as--Tartaglia \cite[Proposition 2.2.5]{ctartaglia} showed that the assembly map for the theory $\HP(\IC -)$ is always injective. Their result also holds for $\PHC(\IC -)$, and this is of course compatible with the above diagram which gives an alternative computation of the assembly map for $\PHC(\IC -)$.
\qed
\end{rem}

\section{Statement and discussion of the conjectures}

\subsection{Burghelea conjecture}

\begin{defn}[Burghelea conjecture {\cite[Section IV]{burghelea}}]
We will say that $G$ satisfies the Burghelea conjecture if the group $\bigoplus_{[g] \in \langle G \rangle^{\infty}} T^G_\ast(g;\IC)$ defined by \eqref{eq:sdf23fds} vanishes.
\qed
\end{defn}

One can state the conjecture for any commutative coefficient ring $R$ with unit. The statement for $\IQ$-coefficients is equivalent to the one for $\IC$-coefficients, and in the literature it is usually stated with $\IQ$-coefficients.

Because of Remark~\ref{remjkds2323} the Burghelea conjecture is the same as the Farrell--Jones conjecture for $\PHC(\IC -)$.

\begin{rem}
The Burghelea conjecture is known to be false without further assumptions on the group. A counter-example was constructed by Burghelea \cite[Section IV]{burghelea}.

So the conjecture is actually about the problem of deciding which classes of groups satisfy it, resp., about constructing more counter-examples. Burghelea himself conjecture that his conjecture should be true for groups admitting a finite classifying space.
\qed
\end{rem}

\begin{rem}\label{jkdsf9023}
We will collect now the known results about the Burghelea conjecture with $\IQ$-coefficients (some of the results hold for more general coefficients than $\IQ$, but we will not mention this in our breakdown).

\begin{enumerate}
\item Burghelea himself noted \cite[Section IV]{burghelea} that his conjecture is true for every fundamental group of a compact, negatively curved Riemannian manifold.

\item Eckmann showed \cite{eckmann} that the Burghelea conjecture is true for every group~$G$ with $\hd_\IQ(G) < \infty$ and belonging to one of the following classes: nilpotent groups, torsion-free solvable groups, linear groups over fields of characteristic $0$, and groups of rational cohomological dimension $\le 2$.\footnote{By a result of Chiswell \cite{chiswell} $1$-relator groups have rational cohomological dimension bounded from above by $2$. Examples of $1$-relator groups are the Baumslag--Solitar groups.} The latter implies that the Burghelea conjecture is true for all groups satisfying \Poincare duality of dimension $2$ over $\IQ$ or of dimension $3$ over $\IZ$ (Eckmann \cite[Theorem~8]{eckmann_2}).

\item Ji showed \cite[Theorem 4.3]{ji_nilpotency} that groups of polynomial growth, arithmetic groups and hyperbolic groups satisfy the Burghelea conjecture.

Furthermore, he proved some closure properties of the class of groups that he investigated (groups from this class satisfy the Burghelea conjecture and this class contains the aforementioned classes of groups).

\item Closure properties of classes of groups defined by properties implying the Burghelea conjecture were studied by several authors. From the results of Chadha and Passi \cite{chadha_passi} in combination with Eckmann's above mentioned results it follows that, e.g., elementary amenable groups of finite rational homological dimension (and so, especially, polycyclic-by-finite groups) satisfy the Burghelea conjecture. Further related results were established in \cite{emma_invent}, \cite{emma_passi_Bass}, \cite{emm_passi} and \cite{sykiotis}.\label{fsd23ds}
\end{enumerate}
To our surprise the above mentioned results seem to be the only ones about the Burghelea conjecture up to now.
\qed
\end{rem}

\begin{rem}
We might ask about invariance of the Burghelea conjecture under quasi-isometries of groups.

Sauer \cite{sauer} showed that the rational homological dimension is a quasi-isometry invariant of amenable groups, and it is conjectured that this should be true in general. But besides this result it seems that everything else related to the Burghelea conjecture is not invariant under quasi-isometries.

What we need for the Burghelea conjecture is finiteness of the homological dimension for the trivial $\IC G$-module $\IC$. But this is not known to be a quasi-isometry invariant.

Furthermore, for the Burghelea conjecture we have to know the structure of the reduced centralizers $C_g/\langle g \rangle$. Now if $G$ and $H$ are quasi-isometric, then the naiv guess would be to say the following: for all $g \in G$ of infinite order there exists an $h \in H$ of infinite order such that $C_g(G)$ and $G_h(H)$ are quasi-isometric. But this is false in general: Raghunathan \cite[Remark 2.15]{raghunathan} constructed two quasi-isometric groups with essentially different centralizers.
\qed
\end{rem}

\subsection{Bass conjecture}\label{secj23223}

\begin{defn}[Bass conjecture for $\calS G$]
We will say that $G$ satisfies the Bass conjecture for $\calS G$ if the image of the map $\trace \circ \chSG_\ast \colon K_\ast^\alg(\calS G) \to \PHC_\ast(\IC G)$ is contained in the first summand $\bigoplus_{[g] \in \langle G \rangle^{\fin}} H_{[\ast]}(C_g;\IC)$ of \eqref{eq:sdfjnk43}.
\qed
\end{defn}

\begin{rem}
>From Diagram \eqref{eq:dsf8934} it follows that we have the following implications for a fixed group $G$:
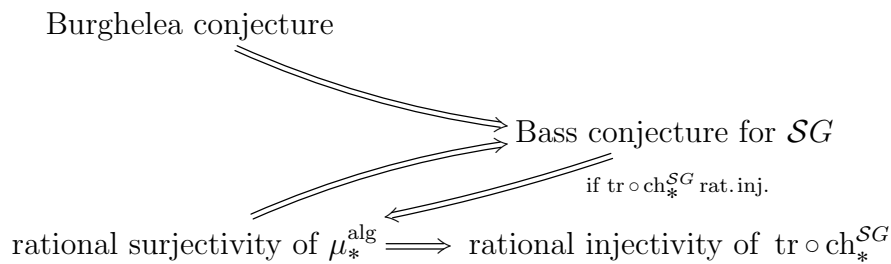
\[\xymatrix{\text{Burghelea conjecture} \ar@{=>}@/_1pc/[dr] &\\
& \text{Bass conjecture for } \calS G \ar@{=>}@/^/[dl]^<<<{\qquad\quad \text{if\,}\trace \circ \chSG_\ast\!\! \text{ rat.\!\! inj.}}\\
\text{ rational surjectivity of } \mu_\ast^\alg \ar@{=>}@/^1pc/[ur] \ar@{=>}[r] & \text{ rational~injectivity~of } \trace \circ \ch_\ast^{\calS G}}\]

Note that by Corollary~\ref{sdf23fwd} rational surjectivity of $\mu_\ast^\alg$ is equivalent to the rational Farrell--Jones conjecture for the coefficient ring $\calS$ and the group $G$.
\qed
\end{rem}

\begin{rem}
The following conjecture was called the strong Bass conjecture in \cite{jor} and the $\IC G$-Bass conjecture in \cite{bcm}: the image of $K_0^\alg(\IC G) \to \HH_0(\IC G)$ is contained in the first summand of \eqref{eq:sdfjsdfjk43}. The Bass conjecture for $\calS G$ implies the strong Bass conjecture by commutativity of the following diagram:
\[\xymatrix{
K_0^\alg(\IC G) \ar[r] \ar[d] & \HH_0(\IC G) \ar@{=}[r] & \HC_0(\IC G)\\
K_0^\alg(\calS G) \ar[rr] & & \PHC_0(\IC G) \ar[u]}\]
The strong Bass conjecture in turn implies the classical Bass conjecture \cite{bass} by a result of Linnell \cite[Lemma~4.1]{linnell_bass} together with the discussion in \cite[Section~3]{eckmann}.

The proof of \cite[Proposition~6.1]{bcm} also applies to the $\IC G$-Bass conjecture showing that it implies the idempotent conjecture for $\IC G$. In fact, the cited arguments show that the $\IC G$-Bass conjecture implies the conjecture that on $K_0^\alg(\IC G)$ the Kaplansky trace coincides with the augmentation trace, and this implies the idempotent conjecture.

So we have the following implications from the Bass conjecture for $\calS G$:
\[\xymatrix{& \text{classical Bass conjecture}\\
\text{Bass conjecture for }\calS G\text{ in the case }\ast = 0 \ar@{=>}@/^1pc/[ur] \ar@{=>}@/_1pc/[dr] &\\
& \text{idempotent conjecture for }\IC G}\]

The Bass conjecture in its different versions is known for far more classes of groups than the Burghelea conjecture.
\qed
\end{rem}

\begin{rem}
Geoghegan \cite{geoghegan} proved the following geometric interpretation of the classical Bass conjecture: a finitely presented group $G$ satisfies the classical Bass conjecture if and only if every homotopy idempotent self-map on a finite, connected complex with fundamental group $G$ has Nielsen number either zero or one. The latter condition can be reformulated by saying that every homotopy idempotent self-map of a closed, smooth, oriented manifold of dimension at least three and with fundamental group $G$ is homotopic to a map with exactly one fixed point. A survey of this geometric interpretation together with full proofs may be found in \cite{bcm2}.
\qed
\end{rem}

\subsection{Homological dimension}

In this section we will collect some well-known facts about (co-)homological dimension, which we will often use in the proofs in Section~\ref{knjdfs923}.

\begin{defn}
Let $G$ be a group and $R$ a ring with unit.
\begin{itemize}
\item Homological dimension $\hd_R(G)$ is the flat dimension of $R$ as an $RG$-module.
\item Cohomological dimension $\cd_R(G)$ is the projective dimension of $R$ as an $RG$-module.
\end{itemize}
For $R = \IZ$ we will write $\hd(-)$ and $\cd(-)$, i.e., not writing the subscript $-_{\IZ}$.

Furthermore, we define $\gd(G)$ as the geometric dimension of $G$, i.e., the least possible dimension of a CW-complex model for $BG$.
\qed
\end{defn}

\begin{lem}[{\cite[Section 4.1]{bieri}, \cite{eg}}]\label{jnsdu23}
We have the following characterizations of the above defined dimensions:
\begin{itemize}
\item $\hd_R(G) = \sup \{n\colon H_n(G;M) \not= 0 \text{ for some } RG\text{-module } M\}$
\item $\cd_R(G) = \sup \{n\colon H^n(G;M) \not= 0 \text{ for some } RG\text{-module } M\}$
\end{itemize}
We also have $\cd(G) = \gd(G)$ with the only possible exception $\cd(G) = 2$, $\gd(G) = 3$.\footnote{It is an open problem (the so-called Eilenberg--Ganea problem) whether a group $G$ with $\cd(G) = 2$, $\gd(G) = 3$ exists. In \cite[Theorem 8.7]{ganea_vs_whitehead} a group is constructed such that it either has this exotic dimension property or there exists a counter-example to the Whitehead conjecture. The latter would mean that there exists a connected subcomplex of an aspherical $2$-complex, which is not aspherical. See also the discussion in \cite[Remark 8.5.7]{davis}.}
\end{lem}

\begin{lem}[{\cite[Theorem 4.6 \& Proposition 4.9]{bieri}}]\label{lem:efs000}
Let $G$ be a countable group and $R$ a commutative ring with unit. Then
\[\hd_R(G) \le \cd_R(G) \le \hd_R(G)+1.\]
If $S < G$ is a subgroup, then $\hd_R(S) \le \hd_R(G)$ and $\cd_R(S) \le \cd_R(G)$.
\end{lem}

\begin{lem}\label{lem:hd.eq.seq}
Let $0 \to A \to B \to C \to 0$ be a short exact sequence of groups.

Then $\hd_R(B) \le \hd_R(A) + \hd_R(C)$.
\end{lem}

\begin{proof}
The claim follows from the Lyndon--Hochschild--Serre spectral sequences
\begin{equation}
\label{eqjnk23fw}
H_p(C; H_q(A;M)) \Rightarrow H_{p+q}(B;M) \quad \text{and} \quad H^p(C; H^q(A;M)) \Rightarrow H^{p+q}(B;M)
\end{equation}
for every $RG$-module $M$.
\end{proof}

\begin{lem}\label{lemjsd9823ds}
Let $G$ be a finite group. Then $\hd_{\IQ}(G) = 0$.

Let $H < G$ be a subgroup of finite index and let $G$ be countable. Then we have the inequality $\hd_{\IQ}(G) \le \hd_{\IQ}(H) + 1$.
\end{lem}

\begin{proof}
The first assertion is \cite[Corollary III.10.2]{brown} combined with Lemma~\ref{lem:efs000}.

For the second assertion we first get by \cite[Proposition III.10.1]{brown} the inequality $\cd_\IQ(G) \le \cd_\IQ(H)$ and combined with Lemma~\ref{lem:efs000} above we get $\cd_\IQ(G) = \cd_\IQ(H)$. Applying Lemma~\ref{lem:efs000} twice we get $\hd_\IQ(G) \le \cd_\IQ(G) = \cd_\IQ(H) \le \hd_\IQ(H) + 1$.
\end{proof}

\begin{lem}\label{lem:powers}
$\hd_{\IQ}(C_{g^n} / \langle g^n\rangle) < \infty \implies 
\hd_{\IQ}(C_g/\langle g^n \rangle) < \infty \implies 
\hd_{\IQ}(C_g / \langle g\rangle) < \infty$, for all $n \neq 0$.
\end{lem}

\begin{proof}
Assume $\hd_{\IQ}(C_{g^n} / \langle g^n\rangle) < \infty$. We have $C_g < C_{g^n}$, thus $C_g/\langle g^n\rangle < C_{g^n}/\langle g^n\rangle$ and so $\hd_{\IQ}(C_g/\langle g^n\rangle) < \infty$.

There is a short exact sequence $0 \to \IZ/n\IZ \to C_g/\langle g^n\rangle \to C_{g}/\langle g\rangle \to 0$ and it follows from
Lyndon--Hochschild--Serre spectral sequence that $\hd_{\IQ}(C_g/\langle g^n \rangle) = \hd_{\IQ}(C_{g}/\langle g \rangle)$.
\end{proof}

\begin{lem}[{\cite[Theorem 2]{bln}}]\label{sdf23fdjn23}
Let $\EGunder$ denote the classifying space for proper actions of a group $G$.

Then we have $\cd_\IQ(G) \le \gdunder(G)$, where the latter is the least possible dimension of a model for $\EGunder$. Consequently, we have $\hd_{\IQ}(G) \le \gdunder(G)$.
\end{lem}

\subsection{Asymptotic dimension}

If $G$ is a countable group, we equip it with a proper, left-invariant metric. Such metrics always exist (even if the group is not finitely generated) and any two of them are coarsely equivalent. So we may talk about large scale invariants of $G$, e.g., asymptotic dimension.

In this section we are going to state two conjectures about groups of finite asymptotic dimension, which would together imply the Burghelea conjecture for groups of finite asymptotic dimension.

We verify them for almost all classes of groups which are known to have finite asymptotic dimension.

\begin{conj}\label{conj:njdsf89}
Let $G$ be a discrete group of finite asymptotic dimension.
\begin{enumerate}
\item\label{12342dsf} $\hd_\IQ(G) < \infty$.
\item If $G$ is torsion-free, then $\hd(G) < \infty$.
\end{enumerate}
\end{conj}

\begin{rem}
There is a quantitative version of Conjecture~\ref{conj:njdsf89} proposing the inequality $\hd_\IQ(G) \le \asdim(G)$, resp., $\hd(G) \le \asdim(G)$ in the torsion-free case. We won't persue the refined version in this paper since this refinement is not needed for the application to the Burghelea conjecture.

Dranishnikov \cite[Proposition 4.10]{dranish_approach} showed that if $G$ is finitely presented and $BG$ is finitely dominated\footnote{These two conditions together are equivalent to the group $G$ being finitely presented and of type FP \cite[Proposition VIII.6.4 \& Theorem VIII.7.1]{brown}. The latter means that $\IZ$ admits a finite projective resolution over $\IZ G$. Conjecturely, groups admitting a finitely dominated $BG$ also admit a finite one (this follows from the $K$-theoretic Farrell--Jones conjecture).}, then $\cd(G) \le \asdim(G)$ and consequently $\hd(G) \le \asdim(G)$.
\qed
\end{rem}

\begin{rem}
Since we always have the inequality $\hd_\IQ(G) \le \gdunder(G)$, but it is not known whether the converse is also true, one might propose a slightly stronger version of the above conjecture: ``$\text{If } G \text{ has finite asymptotic dimension, then } \gdunder(G) \text{ is also finite.}$''

Note that this version includes the case of torsion-free groups, since if $G$ is torsion-free then we have $\EGunder = EG$.

For many of the classes of groups which we discuss in this paper we even prove this slightly stronger version of the conjecture.
\qed
\end{rem}

Another conjecture about asymptotic dimension that we will pursue in this paper is the following one:

\begin{conj}\label{conjdsf89023}
For $g \in G$ let $C_g < G$ denote its centralizer.

If $G$ has finite asymptotic dimension, then $C_g/\langle g \rangle$ also has finite asymptotic dimension.
\end{conj}

The above conjecture will be verified in this paper for almost all classes of groups which are known to have finite asymptotic dimension.

Combining both conjectures stated in this subsection we immediately get the following:

\begin{cor}\label{cor:sd9323}
Let Conjectures \ref{conj:njdsf89}.\ref{12342dsf} and \ref{conjdsf89023} be true.

Then groups of finite asymptotic dimension satisfy the Burghelea conjecture.
\end{cor}

In the following lemma we list some well known properties of asymptotic dimension of groups and in the lemma afterwards we prove some permanence properties of the class of groups satisfying Conjecture~\ref{conjdsf89023}.

We will use the following lemmata in some of the proofs in Section~\ref{knjdfs923}.

\begin{lem}[{\cite[Sections 13 \& 14]{bell_dran}}]\label{lem:asdim}
Let $G$ be a countable group.

Asymptotic dimension is a coarse invariant. In particular:
\begin{enumerate}
\item Let $G' < G$ be a finite index subgroup. Then $\asdim(G')=\asdim(G)$.
\item Let $1 \to F \to G \to Q \to 1$ be an exact sequence of groups with $F$ a finite group. Then $\asdim(G)=\asdim(Q)$.
\end{enumerate}
Moreover, we have a general extension theorem:
\begin{enumerate}
\setcounter{enumi}{2}
\item Let $1 \to K \to G \to Q \to 1$ be an exact sequence of groups. Then we have $\asdim(G) \leq \asdim(K) + \asdim(Q)$.  
\end{enumerate}
\end{lem}

\begin{lem}\label{lem:perm.con.asdim}
Let $G$ be a countable group.
\begin{enumerate}
\item Let $G$ satisfy Conjecture~\ref{conjdsf89023}. Then every subgroup of $G$ satisfies Conjecture~\ref{conjdsf89023}.
\item Let $G = \Asterisk_i H_i$ be a free product of groups satisfying Conjecture~\ref{conjdsf89023}. 
Then the group $G$ also satisfies Conjecture~\ref{conjdsf89023}.
\item Let $1 \to K \to G \xrightarrow{\zeta} Q \to 1$ be an exact sequence and let the groups $K$ and $Q$ satisfy Conjecture~\ref{conjdsf89023}.
Then $G$ also satisfies Conjecture~\ref{conjdsf89023}.
\end{enumerate}
\end{lem}

\begin{proof}
Showing \textit{Claim 1} is straightforward and \textit{Claim 2} follows directly from the canonical form theorem for elements of free products. 

So let us prove \textit{Claim 3}. Notice that by the assumption the asymptotic dimension of $K$ and $Q$ is finite. Let $g \in G$. There are two cases.

\item[Case 1: $g^k \in K$ for some $k \in \mathbb{Z}\setminus\{0\}$.] 
It is easy to see that the centralizer $C_K(g^k)$ is normal in the centralizer $C_G(g^k)$. Thus we have the short exact sequence
$$
1 \to C_K(g^k) \to C_G(g^k) \to C_G(g^k)/C_K(g^k) \to 1,
$$
which immediately leads to
$$
1 \to C_K(g^k)/\langle g^k \rangle \to C_G(g^k)/\langle g^k \rangle \to C_G(g^k)/C_K(g^k) \to 1.
$$
 
By assumption, the quotient $C_K(g^k)/\langle g^k \rangle$ has finite asymptotic dimension. 
Moreover, $C_K(g^k) = C_G(g^k) \cap K$, thus $C_G(g^k)/C_K(g^k) = C_G(g^k)/(C_G(g^k) \cap K) < Q$.
Using now Lemma~\ref{lem:asdim} twice, we get that $C_G(g^k)/C_K(g^k)$ has finite asymptotic dimension, and so has $C_G(g^k)/\langle g^k \rangle$.
The group $C_G(g^k)/\langle g \rangle$ is a quotient of $C_G(g^k)/\langle g^k \rangle$ with a finite kernel, thus by Lemma~\ref{lem:asdim} we get $\asdim(C_G(g^k)/\langle g \rangle) < \infty$.
By the fact that $C_G(g) < C_G(g^k)$ and again by Lemma~\ref{lem:asdim}, we have that $\asdim(C_G(g)/\langle g \rangle) < \infty$.

\item[Case 2: $g^k \notin K$ for all $k \in \mathbb{Z}\backslash\{0\}$.]
We have the following exact sequence:
$$
C_K(g) \to C_G(g) \xrightarrow{\zeta} C_Q(\zeta(q)).
$$
Note that $\zeta$ is not necessarily onto. Although $g$ is not in $K$, we can still consider the group $C_K(g) := C_G(g) \cap K$ of all elements in $K$ commuting with $g$. 
This exact sequence leads to the following one:
$$
C_K(g)/(\langle g \rangle \cap C_K(g)) \to C_G(g)/\langle g \rangle \xrightarrow{\zeta^\prime} C_Q(\zeta(g))/\langle \zeta(g) \rangle.
$$

Note that by the assumption $\langle g \rangle \cap C_K(g) = 1$. 
Thus $C_G(g) / \langle g \rangle$ is an extension of $C_K(g)$ and $\image(\zeta^\prime)$, which both have finite asymptotic dimension.
\end{proof}

\section{Groups (conjecturally) of finite asymptotic dimension}
\label{knjdfs923}

In this section we are going to varify Conjectures~\ref{conj:njdsf89}, \ref{conjdsf89023} and the Burghelea conjecture (which is by Corollary~\ref{cor:sd9323} a consequence of the former two conjectures) for almost all classes of groups known to have finite asymptotic dimension and for some where this is currently only conjectured.

Note that for many of these classes either one of the conjectures is already known (or both, like in the case of hyperbolic groups). But since the techniques we employ to prove the unknown conjectures also prove the known statements, we give unified proofs.

\begin{defn}[{\cite{chadha_passi}}]
The class $E(\IQ)$ consists of groups $G$ satisfying $\hd_{\IQ}(G) < \infty$ and $\hd_{\IQ}(C_g/\langle g\rangle) < \infty$ for every element $g \in G$ of infinite order.
\qed
\end{defn}

It is obvious that groups from the class $E(\IQ)$ satisfy the Burghelea conjecture, and Conjecture~\ref{conj:njdsf89}.\ref{12342dsf} and \ref{conjdsf89023} together imply that groups of finite asymptotic dimension are in $E(\IQ)$.

The arguments Eckmann used to prove the results cited above in Remark~\ref{jkdsf9023} actually show that the classes of groups he considered lie in $E(\IQ)$. The cited results of Ji in that remark have been actually also proved by Ji in the stronger form that the classes of groups he considered lie in $E(\IQ)$, where he uses in his paper the notation $\mathcal{C}$ for this class.

The reason why it is better to know that a class of groups is contained in $E(\IQ)$ than just varifying the Burghelea conjecture for this class is because Chadha and Passi \cite{chadha_passi} showed that $E(\IQ)$ is closed under taking subgroups, extensions, filtered colimits provided the rational homological dimensions of the groups in the system are uniformly bounded from above, and free products. Ji \cite[Theorem 4.3(5)]{ji_nilpotency} extended these closure properties to amalgamated free products and HNN extensions.

\begin{rem}
Note that Conjecture~\ref{conj:njdsf89}.\ref{12342dsf} and \ref{conjdsf89023} imply that if we have a group $G$ of finite asymptotic dimension, then $C_g$ and $C_g/\langle g\rangle$ also lie in the class $E(\IQ)$. We will not persue this implication in this paper, i.e., we will not check it for the classes of groups we consider in this paper.
\qed
\end{rem}

\subsection{Non-positively curved groups}

In this section we will discuss $\CAT(0)$ groups, systolic groups and relatively hyperbolic groups. Note that for hyperbolic groups, which are of course also of non-positive curvature, all conjectures that we pursue in this paper are already known, which is the reason why we do not discuss hyperbolic groups here.

\subsubsection{\texorpdfstring{$\CAT(0)$}{CAT(0)} groups}
\label{sec8923jknd}

\begin{defn}
We call an isometric action of a group on a metric space geometric if it is proper and co-compact.

We call a group $\CAT(0)$ if it acts geometrically on some $\CAT(0)$ space (we do not assume the space to be finite-dimensional).
\qed
\end{defn}

\begin{thm}
Torsion-free $\CAT(0)$ groups have finite homological dimension.

$\CAT(0)$ groups acting on a finite-dimensional $\CAT(0)$ space and virtually torsion-free $\CAT(0)$ groups have finite rational homological dimension.
\end{thm}

\begin{proof}
Let $X$ be the $\CAT(0)$ space on which $G$ acts geometrically. Note that $X$, being a $\CAT(0)$ space, is a model for $\EGunder$ \cite[Theorem 4.6]{lueck_survey}. In the case that $G$ is torsion-free, properness of the action of $G$ implies freeness, and therefore $X$ is a model for $EG$.

So if $X$ is finite-dimensional we are done: either $G$ has torsion and we apply Lemma~\ref{sdf23fdjn23}, or $G$ is torsion-free and we apply the second part of Lemma~\ref{jnsdu23}.

If $X$ is not finite-dimensional, but the group $G$ is torsion-free, we have to use a trick that we learned from a paper of Moran \cite[Proof of Theorem B]{moran}. The quotient $X/G$ is a compact ANR, i.e., for every embedding $i\colon X/G \to Y$ into a metrizable topological space Y, there exists a neighborhood $U$ of $i(X)$ in $Y$ and a retraction of $U$ onto $i(X)$. Hence, by a result of West \cite[Corollary 5.3]{west}, it is homotopy-equivalent to a finite complex $K$. Since $X/G$ is a model for $BG$, $K$ is also a model for $BG$ which is now finite.

If $H$ is a finite index subgroup of $G$, then $H$ is also a $\CAT(0)$ group because it acts geometrically on the same space $X$ on which $G$ acts. So if $H$ is torsion-free, then it has finite homological dimension by the above paragraph. Applying Lemma~\ref{lemjsd9823ds} we conclude that the rational homological dimension of $G$ is finite (by \cite[Theorem III.$\Gamma$.1.1]{bridson_haefliger} we know that $\CAT(0)$ groups are finitely presented and therefore countable).
\end{proof}

\begin{rem}
The cases in the above theorem a priori do not cover all possible cases: it might be that there exists a $\CAT(0)$ group which does not act on any finite-dimensional $\CAT(0)$ space and which is not virtually torsion-free.

It is currently an open problem whether every $\CAT(0)$ group acts on a finite-dimensional $\CAT(0)$ space.
\qed
\end{rem}

\begin{rem}
It is an open problem whether $\CAT(0)$ groups have finite asymptotic dimension. The best result so far is by Wright \cite{wright_cube} who proved that groups acting on a finite-dimensional $\CAT(0)$ cube complex have finite asymptotic dimension.
\qed
\end{rem}

\begin{prop}
Let $G$ be a $\CAT(0)$ group and let $g \in G$ have infinite order.

Then the reduced centralizer $C_g/\langle g \rangle$ virtually embeds as a subgroup into $C_g$.
\end{prop}

\begin{proof}
Let us first show that the centralizer $C_g$ is finitely generated since we will need this later. By \cite[Corollary III.$\Gamma$.4.8]{bridson_haefliger} we know that $\CAT(0)$ groups are semi-hyperbolic. By \cite[Proposition III.$\Gamma$.4.14]{bridson_haefliger} centralizers are quasi-convex (note that the assumption on finite generation is satisfied since $G$ is a $\CAT(0)$ group \cite[Theorem III.$\Gamma$.1.1]{bridson_haefliger}), and so by \cite[Proposition III.$\Gamma$.4.12]{bridson_haefliger} centralizers are finitely generated, quasi-isometrically embedded in $G$ and semi-hyperbolic (by \cite[Theorem III.$\Gamma$.4.9(1)]{bridson_haefliger} the centralizers are even finitely presented).

So $C_g$ is a finitely generated group acting by isometries on some $\CAT(0)$ space $X$, and $\langle g \rangle \cong \IZ$ is central in $C_g$ and acts freely. Recall \cite[Definition II.6.3]{bridson_haefliger} that an isometry of a metric space is either elliptic, hyperbolic or parabolic. By \cite[Proposition II.6.10]{bridson_haefliger} no element of $G$ can act parabolically. We know from the previous paragraph that $C_g$ is quasi-isometrically embedded in $G$ which is in turn quasi-isometric to $X$. So the orbits of $\langle g \rangle$ in $X$ are quasi-isometric to $\IZ$ and we can conclude by \cite[Theorem II.6.8(1)]{bridson_haefliger} that each non-trivial element of $\langle g \rangle$ acts hyperbolically. By \cite[Theorem II.6.12]{bridson_haefliger} we conclude that there exists a subgroup $H < C_g$ of finite index that contains $\langle g \rangle$ as a direct summand, i.e., $H \cong \langle g \rangle \oplus H/\langle g \rangle$. So $C_g / \langle g \rangle$ contains the finite index subgroup $H / \langle g \rangle$ which embeds as a subgroup into $H < C_g$.
\end{proof}

\begin{cor}
Let $G$ be a $\CAT(0)$ group.

If $\hd_\IQ(G) < \infty$, then $G \in E(\IQ)$, and if $\asdim(G) < \infty$, then $\asdim(C_g / \langle g \rangle) < \infty$.
\end{cor}

\begin{cor}
Coxeter groups and right-angled Artin groups are in $E(\IQ)$.
\end{cor}

\begin{proof}
A Coxeter group acts on its Davis complex and a right-angled Artin group acts on the universal cover of its Salvetti complex. Both actions are geometric and both the Davis and the Salvetti complexes are finite-dimensional $\CAT(0)$ complexes (the Salvetti complex is even a cube complex). 
\end{proof}

\subsubsection{Systolic groups}

\begin{defn}[{\cite[Section 2.2]{osaj_pru}}]
A systolic groups is a group which acts geometrically (in this case this means properly, cocompactly and by simplicial automorphisms) on a systolic complex. The latter is a finite-dimensional, uniformly locally finite, connected and simply-connected simplicial complex in which the link of every simplex is $6$-large.
\qed
\end{defn}

\begin{thm}
Systolic groups have finite rational homological dimension.

Torsion-free systolic groups have finite homological dimension.
\end{thm}

\begin{proof}
Chepoi--Osajda \cite[Theorem E]{chepoi_osaj} showed that the systolic complex on which a systolic group $G$ acts is a model for $\EGunder$. Since we assume our systolic complexes to be finite-dimensional, we get $\gdunder(G) < \infty$ and Lemma~\ref{sdf23fdjn23} gives then $\hd_\IQ(G) < \infty$.

(That systolic groups admit finite-dimensional models for $\EGunder$ was first proven by Przytycki \cite{przytycki}.)

If $G$ is torsion-free, then $\EGunder = EG$ and we get $\hd(G) < \infty$.
\end{proof}

\begin{rem}
It is currently not known whether systolic groups have finite asymptotic dimension (but it is conjectured to be so).
\qed
\end{rem}

\begin{thm}
Let $G$ be a systolic groups. Then $G \in E(\IQ)$ and therefore $G$ satisfies the Burghela conjecture.

Furthermore, $\asdim(C_g / \langle g \rangle) < \infty$ for every infinite order element $g \in G$, regardless of whether $G$ itself has finite asymptotic dimension.
\end{thm}

\begin{proof}
Osajda--Prytu\l{}a \cite[Theorem B]{osaj_pru} showed that centralizers of infinite order elements in systolic groups are commensurable with $F_n \times \IZ$ for some $n \in \IN_0$. So we get $\hd_\IQ(C_g / \langle g \rangle) = \asdim(C_g / \langle g \rangle) \le 1$.
\end{proof}

\subsubsection{Relatively hyperbolic groups}

We will be using Osin's definition of relatively hyperbolic groups \cite[Definition 1.6]{osin_memoirs}. If the relatively hyperbolic group is finitely generated, then the number of peripheral subgroups is finite and each of them is also finitely generated \cite[Theorem 1.1]{osin_memoirs}. In this case the definition of Osin coincides with the ones investigated by Bowditch \cite{bowditch_rel} and Farb \cite{farb} (especially, finitely generated relatively hyperbolic groups in Osin's sense satisfy the bounded coset penetration property of Farb).

Recall \cite[Definition 4.1]{osin_memoirs} that an element is called parabolic if it is contained in a conjugate of a peripheral subgroup, and otherwise the element is called hyperbolic.

\begin{lem}\label{lemiu23fd}
Let $G$ be finitely generated and hyperbolic relative to $(H_1, \ldots, H_k)$ and let $g \in G$ be an element of infinite order.

\begin{itemize}
\item If $g$ is hyperbolic, then the centralizer $C_g(G)$ is virtually cyclic.
\item If $g$ is parabolic, i.e., $g = x h_i x^{-1}$ for $x \in G$, $h_i \in H_i$, then $C_g(G) = x C_{h_i}(H_i) x^{-1}$.
\end{itemize}
\end{lem}

\begin{proof}
Let $g$ be hyperbolic. By \cite[Theorem 4.19]{osin_memoirs} we know that $C_g(G)$ is a strongly relatively quasi-convex subgroup of $G$, and by \cite[Theorem 4.16]{osin_memoirs} we know that such subgroups are itself hyperbolic. Since in hyperbolic groups centralizers are virtually cyclic, we conclude that $C_g(G)$ is virtually cyclic.

Let $g$ be parabolic (i.e., $g = x h_i x^{-1}$ for elements $x \in G$ and $h_i \in H_i$). Then we have $C_g(G) = C_{x h_i x^{-1}}(G) = x C_{h_i}(G) x^{-1}$. We will show now $C_{h_i}(G) = C_{h_i}(H_i)$.

Let $p \in C_{h_i}(G)$. Conjugating $C_{h_i}(H_i)$ by $p$ we get $h_i \in p C_{h_i}(H_i) p^{-1} \subset p H_i p^{-1}$. By the almost mal-normality condition \cite[Proposition 2.36]{osin_memoirs} we know that $p H_i p^{-1} \cap H_i$ can only be infinite if $p \in H_i$. Since $h_i$ has infinite order, we conclude $p \in H_i$.
\end{proof}

\begin{cor}
Let $G$ be a finitely generated, relatively hyperbolic group.

If the peripheral subgroups satisfy the Burghelea conjecture, then so does $G$.
\end{cor}

\begin{thm}
Let $G$ be a finitely generated, relatively hyperbolic group.

If the peripheral subgroups of $G$ have finite asymptotic dimension, then $G$ also has finite asymptotic dimension.

If the reduced centralizers of the peripheral subgroups have finite asymptotic dimension, then so do the reduced centralizers of $G$.

Therefore, if the peripheral subgroups satisfy Conjecture~\ref{conjdsf89023}, then so does $G$.
\end{thm}

\begin{proof}
The first statement is due to Osin \cite[Theorem 2]{osin}. Note that this statement does in general not hold true for weakly relatively hyperbolic groups (i.e., dropping the bounded coset penetration property in the definition of relatively hyperbolic groups) \cite[Proposition 3]{osin}.

The second statement follows immediately from the above Lemma~\ref{lemiu23fd}, and the third statement follows from the first and second one.
\end{proof}

\begin{rem}
We would like to say that if the peripheral subgroups all lie in the class $E(\IQ)$, then so does $G$. By Lemma~\ref{lemiu23fd} we at least know that if the reduced centralizers of the peripheral subgroups all have finite rational homological dimension, then so do the reduced centralizers of $G$.

But the problem is that we do not know whether $G$ has finite rational homological dimension provided all of its peripheral subgroups have. By Dahmani \cite{dahmani} we know that a torsion-free group hyperbolic relative to a collection of groups admitting a finite classifying space admits itself a finite classifying space, and that $G$ admits a finite-dimensional $\EGunder$ provided all its peripheral subgroups do. But although we always have $\hd_{\IQ}(G) \le \gdunder(G)$, it is not known if finiteness of $\hd_{\IQ}(G)$ implies finiteness of $\gdunder(G)$.
\qed
\end{rem}

Since hyperbolic groups are hyperbolic relative to the trivial subgroup, we immediately conclude from the above discussion the following corollary. All of these properties of hyperbolic groups are already known for some time. Furthermore, many of the proofs of the properties of relatively hyperbolic groups that we needed above rely on the fact that we already know the corresponding statements for hyperbolic groups. So technically speaking the following is not a corollary of all of the above, but already an ingredient in the proofs of the above statements.

\begin{cor}
Hyperbolic groups have finite asymptotic dimension, satisfy both Conjectures~\ref{conj:njdsf89} and~\ref{conjdsf89023}, and lie in the class $E(\IQ)$ and hence satisfy the Burghelea conjecture.
\end{cor}

\subsection{One-relator groups}

Matsnev \cite{matsnev} proved that one-relator groups always have finite asymptotic dimension.

Chiswell \cite{chiswell} proved that the (rational) cohomological dimension of one-relator groups is always bounded from above by $2$. By L{\"u}ck \cite[Section 4.12]{lueck_survey} we know that we even always have a $2$-dimensional model for $\EGunder$.

Now we have to understand centralizers in one-relator groups. If the group has torsion, we know by Newman \cite[Theorem 2]{newman} that centralizers of any non-trivial element are cyclic, and therefore the reduced centralizers vanish.

In the case of torsion-free one-relator groups we know by Karrass--Pietrowski--Solitar \cite{kps} that every finitely generated subgroup $H$ of a centralizer $C_g$ is of the form $H \cong F_n \times \IZ$ for some $n \ge 0$. Therefore the homological dimension of $H/\langle g\rangle$ is bounded from above by $1$, and hence $\hd(C_g/\langle g \rangle) \le 1$. The same is true for asymptotic dimension, i.e., $\asdim(C_g/\langle g \rangle) \le 1$.

\begin{cor}
One-relator groups have finite asymptotic dimension, satisfy both Conjectures \ref{conj:njdsf89} and \ref{conjdsf89023}, and lie in the class $E(\IQ)$ and satisfy the Burghelea conjecture.
\end{cor}

\subsection{Elementary amenable groups of finite Hirsch length}
\label{secnjkwe9023}

Hillman \cite[Theorem 1.8]{hillman} generalized the notion of Hirsch length from solvable groups to elementary amenable groups. The properties we need are: (i) if $H$ is a subgroup of $G$, then $h(H) \le h(G)$, and (ii) if $H$ is a normal subgroup of $G$, then $h(G) = h(H) + h(G/H)$. Note that the class of elementary amenable groups is, by definition, closed under taking subgroups and quotients.

Finn-Sell--Wu \cite[Theorem 2.12]{finnsellwu} proved that the asymptotic dimension of an elementary amenable group is bounded from above by its Hirsch length. Dranishnikov--Smith \cite[Theorem 3.5]{ds} proved that for virtually polycyclic groups equality holds, and it is an open problem whether equality holds for all elementary amenable groups. But nevertheless, since the Hirsch length of reduced centralizers is bounded by the Hirsch length of the group, we conclude that elementary amenable groups of finite Hirsch length satisfy Conjecture~\ref{conjdsf89023}.

Flores--Nucinkis \cite[Theorem 1]{flores_nun} proved that for an elementary amenable group~$G$ of finite Hirsch length we have $\underline{\operatorname{hd}}(G) = h(G)$. Without defining $\underline{\operatorname{hd}}(G)$ we just note that we always have $\hd_{\IQ}(G) \le \underline{\operatorname{hd}}(G) \le \gdunder(G)$. The proof of this is analogous to the corresponding cohomological statement \cite[Theorem 2]{bln}. Also, Bridson--Kropholler \cite[Theorem I.2]{bridson_kropholler} prove that if $G$ has no torsion, then $\hd(G) \le h(G)$. This shows that elementary amenable groups of finite Hirsch length satisfy Conjecture~\ref{conj:njdsf89} and lie in the class $E(\IQ)$ (but the latter was already known, see Remark~\ref{jkdsf9023}.\ref{fsd23ds}).

Summarizing all the above we get the following theorem:

\begin{thm}
Elementary amenable groups of finite Hirsch length have finite asymptotic dimension, satisfy both Conjectures \ref{conj:njdsf89} and \ref{conjdsf89023}, and lie in the class $E(\IQ)$ and hence satisfy the Burghelea conjecture.
\end{thm}

Note that Hillman--Linnell \cite{hillmanlinnell} showed that elementary amenable groups of finite Hirsch length are locally-finite by virtually-solvable. So we could have deduced the above corollary also from the corresponding, already known statements for solvable groups.

\subsection{Mapping class groups}

Let $S = S^s_{g,r}$ be the compact orientable surface of genus $g$, with $r$ boundary components and with $s$ punctures (removed points).
By $Homeo^+(S,\partial S)$ we denote the group of orientation preserving homeomorphisms of $S$
which fix the punctures and the boundary of $S$ pointwise. 
The mapping class group of $S$ is defined by $MCG(S) = \pi_0(Homeo^+(S,\partial S))$.
In general $MCG(S)$ does not belong to any of classes of groups we study in this paper. 
In particular mapping class groups are neither linear, hyperbolic nor $\CAT(0)$.  

The strategy to prove that $MCG(S) \in E(\IQ)$ and that $MCG(S)$ satisfies Conjecture~\ref{conjdsf89023} is as follows.
First we reduce the problem to surfaces with no punctures.
This is done just for the convenience and is used only in the proof of Lemma~\ref{L:pA}. 
Then, using the canonical forms for elements in mapping class groups, 
we show that the reduced centralizer has a finite index subgroup which fits into a short exact sequence 
with the peripheral groups having finite rational homological dimension and finite asymptotic dimension.

In order to reduce the general case to the case with no punctures we use the Birman exact sequence:

$$
1 \rightarrow \pi_1(S^{s-1}_{g,r}) \xrightarrow{Push} MCG(S^s_{g,r}) \xrightarrow{F} MCG(S^{s-1}_{g,r}) \rightarrow 1.
$$
(Note that this exact sequence does not apply when $S^s_{g,r}$ is the one punctured torus. Then we have $MCG(S^1_{0,0}) = MCG(S^0_{0,0})$. 
The reason is that the center of the fundamental group of the torus is non-trivial.)

The map $F$ is induced by inclusion to the surface with one puncture less (for the definition of the push map consult \cite[Theorem 4.6]{farb.margalit}).
 
Fundamental group of a surface belongs to $E(\IQ)$ and satisfies Conjecture~\ref{conjdsf89023}. 
Thus by induction and the fact that $E(\IQ)$ 
and the class of groups satisfying Conjecture~\ref{conjdsf89023} are closed under extensions (\cite{chadha_passi} and Lemma~\ref{lem:perm.con.asdim}),
it is enough to deal with surfaces with no punctures.

>From now on we assume that $S$ has no punctures. We will need the following three lemmata for the final proof.

\begin{lem}
$MCG(S)$ has finite rational homological dimension, and if it is torsion-free then it has finite homological dimension.
\end{lem}

\begin{proof}
If $S$ has negative Euler characteristic then both claims follow from Harer' results \cite[Theorem 4.1]{harer}.
For non-negative Euler characteristic we have only two cases where the mapping class group is infinite, namely the torus and the annulus: we have $MCG(S^0_{1,0})=SL_2(\IZ)$, $MCG(S^0_{0,2}) = \IZ$ and both have finite (rational) homological dimension.
\end{proof}

Note that in this section we take the liberty to use a slightly different notation: instead of $C_g/\langle g \rangle$ we write $C(g)/g$. 

\begin{lem}\label{L:pA}
Let $g \in MCG(S)$ be a pseudo-Anosov element. Then $C(g)/g$ is finite.
\end{lem}

\begin{proof}
In \cite{mccarthy} it is proved that for $S$ with no punctures, $C(g)$ is virtually infinite cyclic. 
Since $g$ is not torsion, the quotient $C(g)/g$ is finite. 
\end{proof}

%There is finitely many compact orientable surfaces of positive Euler characteristic. 
%In this case it is well know that their mapping class groups belong to the class $E(\IQ)$.
%Indeed, the only two infinite cases are $MCG$ of the torus 
%which is isomorphic to $SL_2(\IZ)$ and $MCG$ of the annulus which is isomorphic to $\IZ$.

%For the rest of this section let us assume that Euler characteristic of $S$ is negative. \\
%Then the mapping class group of $S$ acts properly discontinously on its Teichmuller space. 
%Since the Teichmuller space is homeomorphic to a disc, $MCG(S)$ has finite rational homological dimension (\cite{harer}).
 
Now we recall briefly the classification of elements in $MCG(S)$. 
For an extensive discussion consult \cite{farb.margalit}.

%Then there are three possibilities (not necessary disjoint): either
%$g$ is torsion, pseudo-Anosov or reducible. 
%Reucible elements are defined below. For the definition of pseudo-Anosov elements see \cite[13.2]{farb.margalit}.

%\begin{lem} 
%If $g \in MCG(S)$ is pseudo-Anosov, then $C(g)/g$ is finite.
%\end{lem}

%\begin{proof}
%.
%\end{proof}

Let $c_1$ and $c_2$ be two isotopy classes of simple loops in $S$. 
By $i(c_1,c_2)$ we denote the geometric intersection number, i.e., $\min\{|\gamma_1 \cap \gamma_2| \colon \gamma_i~\text{represents}~c_i~\text{for}~i=1,2\}$. 
A simple loop is called \textbf{essential} if it is not contractible and not homotopic to a boundary loop. 
An element $g \in MCG(S)$ is called \textbf{reducible} if there is a non-empty set $C = \{c_1,\ldots,c_n\}$ of isotopy classes
of essential simple loops in $S$ so that $i(c_i,c_j)=0$ for all $i$ and $j$ and which is $g$-invariant (as a set).
Such a set $C$ is called a reduction system for $g$.
  
The reduction system for $g$ is maximal if it is not a proper subset of any other reduction set for $g$. 
Note that there may be many maximal reduction systems. However, we can define the unique reduction system by taking the intersection of all maximal reduction systems. This intersection is called the \textbf{canonical reduction system}. 
The canonical reduction system is not necessary maximal. 
By convention, if the element $g \in MCG(S)$ is not reducible, then the canonical reduction system for $g$ is empty.

\begin{lem}\label{lem:crs.fixed}
Let $g \in MCG(S)$ and let $CRS_g$ be its canonical reduction system. Let $h \in C(g)$. Then $h(CRS_g) = CRS_g$.
\end{lem}
\begin{proof}
It is easy to see that $h$ acts on the set of all maximal reduction systems. 
Since $CRS_g$ is the intersection of all of them, it is fixed by $h$. 
\end{proof}

Now let us describe the canonical form of an element of a mapping class group. 
Let $g \in MCG(S)$ and let $CRS_g = \{c_1,\ldots,c_m\}$ be its canonical reduction system.
Choose pairwise disjoint representatives of the classes $c_i$ together with pairwise disjoint closed annuli $R_1,\ldots,R_m$, 
where $R_i$ is a closed tubular neighborhood of the representative of $c_i$.
Let $R_{m+1},\ldots,R_{m+p}$ be closures of connected components of $S - \bigcup_{i=1}^m R_i$.
Then there is a representative $\psi \in Homeo^+(S)$ of $g$ and $k$ such that $\psi^k(R_i) = R_i$ for $ 1 \leq i \leq m+p$.
Moreover, if $\psi^k_{R_i}$ denotes the restriction of $\psi^k$ to $R_i$, then $\psi^k_{R_i}$ is a power of Dehn twist for $1 \leq i \leq m$ and
pseudo-Anosov or identity for $m+1 \leq i \leq m+p$.

Let us now use this decomposition to give a description of the centraliser. 
Let $h \in C(g)$. The elements of $CRS_g$ are permuted by $h$. 
Let $C_o(g)$ be the finite index subgroup of $C(g)$ consisting of elements which fix $CRS_g$ pointwise. 
If $h \in C_o(g)$, then there exists a representative $\psi_h$ of $h$, such that $h(R_i) = R_i$ for $1 \leq i \leq m$.
Then $\psi_h$ permutes $R_i$ for $m+1 \leq i \leq m+p$. 
Let $C_{oo}(g)$ be the finite index subgroup of $C_o(g)$ consisting of those elements $h$ which have a representative $\psi_h$ 
such that $\psi_h(R_i) = R_i$ for all $i$. 
Note that from the definition of $k$ from the previous paragraph, $g^k \in C_{oo}(g)$. 

We have the cutting homomorphism (\cite[Prop.3.20]{farb.margalit})
$$
C_{oo}(g) \xrightarrow{\zeta_i} MCG(R_i),
$$
where $\zeta_i(h)$ is the isotopy class of the restriction of $\psi_h$ to $R_i$. 
Let $\psi^k_g$ be a homeomorphism representing $g^k$ and fixing all the subsurfaces $R_i$. 
Denote by $g^k_{R_i}$ the isotopy class of $\psi^k_g$ restricted to $R_i$.
We have that $\zeta_i(g^k) = g^k_{R_i}$. 
Moreover, if $h$ commutes with $g$, then it commutes with $g^k$ and therefore $\zeta_i(h)$ commutes with $\zeta_i(g^k) = g^k_{R_i}$.
Thus $\zeta_i$ ranges in $C(g^k_{R_i}) < MCG(R_i)$.
If we sum up all these cutting homomorphism, we get an inclusion
$$
C_{oo}(g) \xrightarrow{\phantom{a}\zeta\phantom{a}} \bigoplus_{i=1}^{i=m+p} C(g^k_{R_i}).
$$

Indeed, if $\zeta(h)$ is trivial, it means that it is isotopic to the identity separately on each annulus $R_i$. 
These isotopies fix the boundaries of $R_i$ pointwise, thus one can compose them together to an 
isotopy on $S$.

\begin{thm}
Let $S = S^s_{g,r}$. Then $MCG(S)$ belongs to the class $E(\IQ)$.

Moreover, $MCG(S)$ satisfies Conjecture~\ref{conjdsf89023}.
\end{thm}

\begin{proof}

We keep the notation from the discussion preceding the theorem.
By the reduction step we can assume that $S$ has no punctures.
 
Note that each element $g^k_{R_i}$ (seen as an element of $MCG(S)$) belongs to $C_{oo}(g)$. 
Consider the inclusion $\zeta \colon C_{oo}(g) \to \bigoplus_{i=1}^{i=m+p} C(g^k_{R_i})$.
If we reduce each centralizer in the sum on the right we obtain the following exact sequence:

$$
1 \to \langle g^k_{R_1},\ldots, g^k_{R_{m+p}} \rangle \to C_{oo}(g) \xrightarrow{\zeta'} \bigoplus_{i=1}^{i=m+p} C(g^k_{R_i})/g^k_{R_i}.
$$

Since $g^k = g^k_{R_1} \ldots g^k_{R_{m+p}}$, it follows that $\zeta'(g^k) = 1$. We finally get the following:

$$
1 \to \langle g^k_{R_1},\ldots,g^k_{R_{m+p}} \rangle/g^k \to C_{oo}(g)/g^k \xrightarrow{\zeta''} \bigoplus_{i=1}^{i=m+p} C(g^k_{R_i})/g^k_{R_i}.
$$

The group $\langle g^k_{R_1},\ldots,g^k_{R_{m+p}} \rangle/g^k$ is abelian, thus has finite rational homological dimension. 
To prove that $\hd_{\IQ}(C_{oo}(g)/g^k)$ is finite, it is enough to prove that for each $i$ the
reduced centralizer $C(g^k_{R_i})/g^k_{R_i}$ has finite rational homological dimension.
Indeed, although $\zeta''$ is not onto, we can substitute $\bigoplus_{i=1}^{i=m+p} C(g^k_{R_i})/g^k_{R_i}$ by the image of $(\zeta'')$ 
and use Lemma~\ref{lem:hd.eq.seq}.

Let us now show that the quotient $C(g^k_{R_i})/g^k_{R_i}$ has finite rational homological dimension.
If $1 \leq i \leq m$, then $R_i$ is the annulus, so $MCG(R_i) = \IZ$ and $C(g^k_{R_i})/g^k_{R_i}$ is finite.
For $m+1 \leq i \leq m+p$ we know that the element $g^k_{R_i}$ is either pseudo-Anosov or identity. 
If it is pseudo-Anosov, then $C(g^k_{R_i})/g^k_{R_i}$ is finite (Lemma~\ref{L:pA}).
If it is the identity, then we have $C(g^k_{R_i})/g^k_{R_i} = MCG(R_i)$.
In all this cases we get a group with finite rational homological dimension.

It means that $C(g)/g^k$ contains a finite index subgroup $C_{oo}(g)/g^k$ whose rational homological dimension is finite. 
Thus finally by Lemma~\ref{lem:powers} and Lemma~\ref{lemjsd9823ds}, $C(g)/g$ has finite rational homological dimension.

The proof of Conjecture~\ref{conjdsf89023} goes along the same lines. The crucial fact to start the arguments is that the mapping class groups have finite asymptotic dimension \cite{mcg.bbf}.
\end{proof}

\subsection{Discrete subgroups of almost connected Lie groups}

We call a Lie group almost connected, if it has finitely many connected components.

Ji \cite[Corollary 3.4]{ji_arithmetic} showed that finitely generated, discrete subgroups of connected Lie groups have finite asymptotic dimension (see Carlsson--Goldfarb \cite[Corollary~3.6]{carlsson_goldfarb} for the special case that cocompact lattices in connected Lie groups have finite asymptotic dimension). Standard techniques allow us to generalize Ji's result:

\begin{thm}[Generalizing Ji {\cite[Corollary 3.4]{ji_arithmetic}}]\label{thm32d23}
Discrete subgroups of almost connected Lie groups have finite asymptotic dimension.
\end{thm}

\begin{proof}
Let us first recall the argument from \cite[Corollary 3.4]{ji_arithmetic} that finitely generated, discrete subgroups of connected Lie groups have finite asymptotic dimension: if $G$ is a connected Lie group and $K$ a maximal compact subgroup in $G$, then the homogeneous space $K \backslash G$ (endowed with a $G$-invariant Riemannian metric) satisfies $\operatorname{as-dim}(K \backslash G) < \infty$. This follows from the proof of \cite[Corollary~3.6]{carlsson_goldfarb}. Now if $\Gamma$ is a finitely generated, discrete subgroups of $G$, then it acts isometrically and properly on $K \backslash G$. Therefore, by \cite[Proposition 2.3]{ji_arithmetic}, we get $\asdim(\Gamma) \le \asdim(K \backslash G)$.

Now let $\Gamma$ be an arbitrary discrete subgroup of $G$. By \cite[Theorem 2.1]{ds} we have
\[\asdim(\Gamma) = \sup_{F \le_{\mathrm{f.g.}} \Gamma} \asdim(F),\]
where the supremum is taken over all finitely generated subgroups $F$ of $\Gamma$. But such an $F$ is a finitely generated, discrete subgroup of $G$, and therefore $\asdim(F) \le \asdim(K \backslash G)$. So we conclude $\asdim(\Gamma) \le \asdim(K \backslash G)$.

If $G$ is almost connected, we may pass to its identity component $G_0$ and also pass to the subgroup $\Gamma \cap G_0$ of $\Gamma$. For $\Gamma \cap G_0$ the above arguments apply and show $\asdim(\Gamma \cap G_0) < \infty$. Therefore $\asdim(\Gamma) < \infty$ since $\Gamma \cap G_0$ is a finite index subgroup of $\Gamma$.
\end{proof}

\begin{thm}\label{thmsdf8912123}
Discrete subgroups of almost connected Lie groups satisfy the Conjecture~\ref{conj:njdsf89}.
\end{thm}

\begin{proof}
In \cite[Section 2]{baum_connes_higson} it is discussed that if $G$ is an almost connected Lie group, then $K \backslash G$ (for $K$ a maximal compact subgroup of $G$) is a model for $\underline{E G}$ (see \cite{abels}). Furthermore, if $\Gamma$ is any discrete subgroup of $G$, then $K \backslash G$ is also a model for $\underline{E \Gamma}$. It follows that the proper geometric dimension $\gdunder(\Gamma)$ of discrete subgroups of almost connected Lie groups is finite. By \cite[Theorem 2]{bln} we know that $\cd_\IQ(\Gamma) \le \gdunder(\Gamma)$ for any group $\Gamma$. With Lemma~\ref{lem:efs000} we therefore get $\hd_\IQ(\Gamma) \le \gdunder(\Gamma)$ and conclude that $\hd_\IQ(\Gamma)$ is finite for discrete subgroups of almost connected Lie groups.

So it remains to treat the torsion-free case. By \cite[Theorem 11]{malcev} (see also \cite[Theorem 6]{iwasawa}) we know that the quotient $K \backslash G$ of an almost connected Lie group $G$ by a maximal compact subgroup $K$ is diffeomorphic to Euclidean space (although in the cited theorems it is stated that it is homeomorphic to Euclidean space, the proofs actually show that it is diffeomorphic to it). If $\Gamma$ is now a discrete, torsion-free subgroup of $G$, then $K \backslash G / \Gamma$ is a smooth manifold whose universal cover is contractible, i.e., it is a model for the classifying space $B \Gamma$. We conclude that $\hd(\Gamma) < \infty$.
\end{proof}

Eckmann showed \cite{eckmann} that the Burghelea conjecture holds for linear groups over fields of characteristic zero and with finite rational homological dimension. The arguments used in the above proof of Theorem~\ref{thmsdf8912123} show that discrete subgroups of almost connected, linear Lie groups satisfy this assumption. So such groups satisfy the Burghelea conjecture. We will show now that we can actually drop the linearity assumption:

\begin{thm}\label{thmaaa003}
Discrete subgroups of almost connected Lie groups lie in the class $E(\IQ)$ and therefore satisfy the Burghelea conjecture.
\end{thm}

\begin{proof}
Let $\Gamma$ be a discrete subgroup of the almost connected Lie group $G$. From the Theorem~\ref{thmsdf8912123} we know that $\Gamma$ has finite rational homological dimension. So it remains to show that the reduced centralizers of $\Gamma$ also have finite rational homological dimension.

Note that we can assume $G$ to be connected. Indeed, if we denote by $G^\circ$ the identity component of $G$, then $\Gamma^\circ := \Gamma \cap G^\circ$ will be of finite index in $\Gamma$. So if $g \in \Gamma$ has infinite order, some power $g^n$ will be contained in $\Gamma^\circ$ and be of infinite order. The arguments below will show that the rational homological dimension of $C_{g^n}(\Gamma^\circ) / \langle g^n\rangle$ is finite, since $\Gamma^\circ$ is a discrete subgroup of the connected Lie group $G^\circ$. Since $\Gamma^\circ$ has finite index in $\Gamma$, $C_{g^n}(\Gamma^\circ)$ will have finite index in $C_{g^n}(\Gamma)$. So $C_{g^n}(\Gamma) / \langle g^n\rangle$ has finite rational homological dimension. We finish the argument with Lemma~\ref{lem:powers}.

So assume that $G$ is connected. The solv-radical $R$ of $G$ is by definition a maximal, connected, solvable, normal Lie subgroup of $G$. It exists \cite[Theorem~5.11 in I.2.§5]{solv_radical}, is unique and a closed subgroup \cite[1st paragraph of XVIII.2]{hochschild}, and has the property that the quotient $G/R$ is a connected, semi-simple Lie group \cite[2nd paragraph after proof of Theorem~5.11]{solv_radical}. We have a short exact sequence
\[0 \to R \to G \to G/R \to 0.\]

If $\Gamma$ is a discrete subgroup of $G$, then $C_g < \Gamma$ is also a discrete subgroup of $G$, where $C_g$ is the centralizer in $\Gamma$ of $g \in \Gamma$. The kernel of the homomorphism $C_g \hookrightarrow G \to G/R$ is exactly $C_g \cap R$ and we therefore get the diagram
\[\xymatrix{
0 \ar[r] & R \ar[r] & G \ar[r] & G/R \ar[r] & 0\\
0 \ar[r] & C_g \cap R \ar[r] \ar[u] & C_g \ar[r] \ar[u] & C_g / (C_g \cap R) \ar[r] \ar[u] & 0}\]

Recall that we want to show $\hd_\IQ(C_g / \langle g \rangle) < \infty$ for $g$ an element of infinite order.

\begin{description}
\item[1st case: $\bm{g \in R}$.] Note that $g$ is central in $C_g$ and therefore also in $C_g \cap R$. We get the following short exact sequence:
\[0 \to (C_g \cap R) / \langle g \rangle \to C_g / \langle g \rangle \to C_g / \langle g \rangle \big/ (C_g \cap R) / \langle g \rangle \to 0.\]
Note that $C_g / \langle g \rangle \big/ (C_g \cap R) / \langle g \rangle \cong C_g / (C_g \cap R)$ which is a discrete subgroup of the connected Lie group $G/R$. The rational homological dimension of $C_g / (C_g \cap R)$ is therefore finite by Theorem~\ref{thmsdf8912123}.

It remains to show that the rational homological dimension of $(C_g \cap R) / \langle g \rangle$ is finite. But $C_g \cap R$ is a discrete subgroup of the connected, solvable Lie group $R$ and therefore we may apply Lemma~\ref{lemjnsdf234t3} which we will prove below.

\item[2nd case: $\bm{g^n \in R}$ for $\bm{n > 1}$, but $\bm{g \notin R}$.] We go through the arguments in the 1st case to conclude that $\hd_\IQ(C_g / \langle g^n \rangle) < \infty$. From this it follows that $\hd_\IQ(C_g / \langle g \rangle) < \infty$.

\item[Remaining case: $\bm{\langle g\rangle \cap R = \emptyset}$.] >From $\langle g\rangle \cap R = \emptyset$ we conclude that the composition $C_g \cap R \to C_g \to C_g / \langle g\rangle$ is injective. Furthermore, $C_g \cap R$ is normal in $C_g / \langle g\rangle$. So we have the following short exact sequence:
\[0 \to C_g \cap R \to C_g / \langle g\rangle \to C_g / \langle g\rangle \big/ (C_g \cap R) \to 0.\]

$C_g \cap R$ is a discrete subgroup of the connected Lie group $R$ and therefore, by Theorem~\ref{thmsdf8912123}, has finite rational homological dimension.

It remains to show that $C_g / \langle g\rangle \big/ (C_g \cap R)$ has finite rational homological dimension. We have $C_g / \langle g\rangle \big/ (C_g \cap R) \cong C_g / (C_g \cap R) \big/ \langle g\rangle$ and $C_g / (C_g \cap R)$ is a discrete subgroup of the connected, semi-simple Lie group $G/R$. So we conclude with Lemma~\ref{lemjns4t3} which we will prove further below.
\end{description}

So the proof of Theorem~\ref{thmaaa003} is now complete since the above three cases cover every possible case.
\end{proof}

\begin{lem}\label{lemjnsdf234t3}
Let $R$ be a connected, solvable Lie group and let $\Gamma < R$ be a discrete subgroup with a central element $g \in \Gamma$ of infinite order.

Then $\hd_{\IQ}(\Gamma / \langle g \rangle) < \infty$.
\end{lem}

\begin{proof}
Since $R$ is a connected, solvable Lie group it has a unique, maximal, normal, connected, nilpotent Lie subgroup \cite[Chapter III.2]{auslander} called the nil-radical of $R$ and denoted by $N$. The quotient $R/N$ is abelian and we have a short exact sequence
\[0 \to N \to R \to R/N \to 0.\]
The kernel of the homomorphism $\Gamma \hookrightarrow R \to R/N$ is exactly $\Gamma \cap N$ and we therefore have the diagram
\[\xymatrix{
0 \ar[r] & N \ar[r] & R \ar[r] & R/N \ar[r] & 0\\
0 \ar[r] & \Gamma \cap N \ar[r] \ar[u] & \Gamma \ar[r] \ar[u] & \Gamma / (\Gamma \cap N) \ar[r] \ar[u] & 0}\]

\begin{description}
\item[1st case: $\bm{g \in N}$.] We have that $g$ is central in $\Gamma \cap N$ since it is central in $\Gamma$, and so we get the short exact sequence
\[0 \to (\Gamma \cap N) / \langle g\rangle \to \Gamma / \langle g\rangle \to \Gamma / \langle g\rangle \big/ (\Gamma \cap N) / \langle g\rangle \to 0.\]
Since we have $\Gamma / \langle g\rangle \big/ (\Gamma \cap N) / \langle g\rangle \cong \Gamma / (\Gamma \cap N)$, which is a discrete subgroup of the connected Lie group $R/N$, we get by Theorem~\ref{thmsdf8912123} that its rational homological dimension is finite.

It remains to show that the rational homological dimension of $(\Gamma \cap N) / \langle g\rangle$ is finite. We have that $\Gamma \cap N$ is a subgroup of $N$ which is nilpotent. So $\Gamma \cap N$ is nilpotent itself, and has finite rational homological dimension by Theorem~\ref{thmsdf8912123} ($\Gamma \cap N$ is a discrete subgroup of the connected Lie group $N$). By \cite[Theorem 2.2]{eckmann} we get that the rational homological dimension of $(\Gamma \cap N) / \langle g\rangle$ is finite.

\item[2nd case: $\bm{g^n \in N}$ for $\bm{n > 1}$, but $\bm{g \notin N}$.] We apply the arguments from the 1st case to the short exact sequence
\[0 \to (\Gamma \cap N) / \langle g^n\rangle \to \Gamma / \langle g^n\rangle \to \Gamma / \langle g^n\rangle \big/ (\Gamma \cap N) / \langle g^n\rangle \to 0\]
and conclude that the rational homological dimension of $\Gamma / \langle g^n\rangle$ is finite. Since the kernel of $\Gamma / \langle g^n\rangle \to \Gamma / \langle g\rangle$ is finite, we conclude that $\hd_\IQ(\Gamma / \langle g\rangle) < \infty$.

\item[Remaining case: $\bm{\langle g\rangle \cap N = \emptyset}$.] >From $\langle g\rangle \cap N = \emptyset$ we conclude that the composition $\Gamma \cap N \to \Gamma \to \Gamma / \langle g\rangle$ is injective. Furthermore, $\Gamma \cap N$ is normal in $\Gamma / \langle g\rangle$, so we get the short exact sequence
\[0 \to \Gamma \cap N \to \Gamma / \langle g\rangle \to \Gamma / \langle g\rangle \big/ (\Gamma \cap N) \to 0.\]

By Theorem~\ref{thmsdf8912123} we know that $\Gamma \cap N$ has finite rational homological dimension.

It remains to show that $\Gamma / \langle g\rangle \big/ (\Gamma \cap N) \cong \Gamma / (\Gamma \cap N) \big/ \langle g\rangle$ has finite rational homological dimension. Now $\Gamma / (\Gamma \cap N)$ is a subgroup of $R/N$ which is an connected, abelian Lie group. So $\Gamma / (\Gamma \cap N)$ itself is abelian and it has finite rational homological dimension by Theorem~\ref{thmsdf8912123}. So by \cite[Theorem 2.2]{eckmann} we get that the rational homological dimension of $\Gamma / (\Gamma \cap N) \big/ \langle g\rangle$ is finite.
\end{description}

This completes the proof that $\hd_{\IQ}(\Gamma / \langle g \rangle) < \infty$ since we have managed to handle all three occuring cases.
\end{proof}

\begin{lem}\label{lemjns4t3}
Let $S$ be a connected, semi-simple Lie group and let $\Gamma < S$ be a discrete subgroup with a central element $g \in \Gamma$ of infinite order.

Then $\hd_{\IQ}(\Gamma / \langle g \rangle) < \infty$.
\end{lem}

\begin{proof}
Let $Z$ denote the center of $S$ and look at the short exact sequence
\[0 \to Z \to S \to S / Z \to 0.\]
The kernel of the map $\Gamma \hookrightarrow S \to S / Z$ is $\Gamma \cap Z$ and so we have the diagram
\[\xymatrix{
0 \ar[r] & Z \ar[r] & S \ar[r] & S / Z \ar[r] & 0\\
0 \ar[r] & \Gamma \cap Z \ar[r] \ar[u] & \Gamma \ar[r] \ar[u] & \Gamma / (\Gamma \cap Z) \ar[r] \ar[u] & 0}\]

\begin{description}
\item[1st case: $\bm{g \in Z}$.] The element $g$ is central in $\Gamma$ and therefore also in $\Gamma \cap Z$. We look at the short exact sequence
\[0 \to (\Gamma \cap Z) / \langle g\rangle \to \Gamma / \langle g\rangle \to \Gamma / \langle g\rangle \big/ (\Gamma \cap Z) / \langle g\rangle \to 0.\]
We have $\Gamma / \langle g\rangle \big/ (\Gamma \cap Z) / \langle g\rangle \cong \Gamma / (\Gamma \cap Z)$ which is a discrete subgroup of the connected Lie group $S / Z$ and therefore has finite rational homological dimension by Theorem~\ref{thmsdf8912123}.

$\Gamma \cap Z$ is abelian (since it is a subgroup of the center of $S$, which is abelian) and it is a discrete subgroup of the connected Lie group $S$. Because of the latter it has by Theorem~\ref{thmsdf8912123} finite rational homological dimension, and by \cite[Theorem 2.2]{eckmann} we conclude that $(\Gamma \cap Z) / \langle g \rangle$ has finite rational homological dimension.

We conclude that $\hd_\IQ(\Gamma / \langle g\rangle) < \infty$.

\item[2nd case: $\bm{g^n \in Z}$ for $\bm{n > 1}$, but $\bm{g \notin Z}$.] We do the arguments from the 1st case for $g^n$ instead of $g$ and conclude that $\Gamma / \langle g^n\rangle$ has finite rational homological dimension.

>From this it follows that $\hd_\IQ(\Gamma / \langle g\rangle) < \infty$.

\item[Remaining case: $\bm{\langle g\rangle \cap Z = \emptyset}$.] The composition $\Gamma \cap Z \to \Gamma \to \Gamma / \langle g\rangle$ is injective and $\Gamma \cap Z$ is a normal subgroup of $\Gamma / \langle g\rangle$. So we have the short exact sequence
\[0 \to \Gamma \cap Z \to \Gamma / \langle g\rangle \to \Gamma / \langle g\rangle \big/ (\Gamma \cap Z) \to 0.\]

$\Gamma \cap Z$ is a discrete subgroup of the connected Lie group $S$ and therefore has finite rational homological dimension by Theorem~\ref{thmsdf8912123}.

We have the isomorphism $\Gamma / \langle g\rangle \big/ (\Gamma \cap Z) \cong \Gamma / (\Gamma \cap Z) \big/ \langle g\rangle$ and $\Gamma / (\Gamma \cap Z)$ is a subgroup of the Lie group $S/Z$. Since $S$ is semi-simple, its center $Z$ is discrete, and so by \cite[Exercise 7.11(b)]{fulton} the Lie group $S/Z$ has trivial center. Since $S$ is connected, the quotient $S/Z$ is also connected. So the adjoint representation of $S/Z$ is faithful, i.e., $S/Z$ is a linear group.

Therefore the subgroup $\Gamma / (\Gamma \cap Z)$ of $S/Z$ is a linear group over a field of characteristic zero, and $\Gamma / (\Gamma \cap Z)$ has finite rational homological dimension by Theorem~\ref{thmsdf8912123}. By Eckmann \cite[Theorem 2.4']{eckmann} we conclude that $\Gamma / (\Gamma \cap Z) \big/ \langle g\rangle$ has finite rational homological dimension.
\end{description}

This completes the proof that $\hd_{\IQ}(\Gamma / \langle g \rangle) < \infty$ since we have managed to handle all three occuring cases.
\end{proof}

\begin{thm}
Discrete subgroups of almost connected Lie groups satisfy the Conjecture~\ref{conjdsf89023}.
\end{thm}

\begin{proof}
Let $\Gamma$ be a discrete subgroup of an almost connected Lie group $G$. Because of Theorem~\ref{thm32d23} we know that $\Gamma$ has finite asymptotic dimension, and therefore any centralizer $C_g$ also has finite asymptotic dimension. Using the arguments from the proof of Theorem~\ref{thmaaa003} we see that in order to show finiteness of asymptotic dimension of the reduced centralizers we can reduce to the cases where $G$ is either solvable or semi-simple.

The solvable case is treated in Section~\ref{secnjkwe9023}. For the semi-simple case we start arguing as in the proof of Lemma~\ref{lemjns4t3}. In the first case of that proof we reduce to the abelian case, i.e., that $\Gamma$ is abelian. That in the abelian case the reduced centralizer also has finite asymptotic dimension now follows from \cite[Corollary 3.3]{ds}.

So we are now in the remaining case of the proof of Lemma~\ref{lemjns4t3}, i.e., the case that $G$ is a connected, semi-simple Lie group with trivial center. In this case we know that the adjoint representation of $G$ is faithful and therefore $G$ is a matrix Lie group, i.e., can be embedded into $\GL(n;\IR)$.

$\GL(n;\IR) / O(n;\IR)$ is a $\CAT(0)$ space \cite[Proposition II.10.33 \& Theorem~II.10.39]{bridson_haefliger}. By \cite[Proposition II.10.61]{bridson_haefliger} we furthermore know that an element $g \in \Gamma$ acts semi-simple if and only if the matrix $g$ is semi-simple, i.e, can be diagonalized in $\GL(n;\IC)$. So if $g$ is semi-simple, then we can finish the argument as in Section~\ref{sec8923jknd} to show that the reduced centralizer $C_g/\langle g \rangle$ has finite asymptotic dimension.

It remains to treat the case that $g$ is a parabolic element. We know that the action of $g$ on the proper $\CAT(0)$ space $X := \GL(n;\IR) / O(n;\IR)$ extends onto its boundary $\partial X$ and since $g$ is parabolic it has at least one fixed point on $\partial X$ \cite[Proposition II.8.25]{bridson_haefliger}. Furthermore, Fujiwara--Nagano--Shioya \cite[Theorem 1.3]{fns} proved that there even exists a fixed point $\xi \in \partial X$ of $g$ that is also a fixed point of the whole centralizer $C_g$. By Caprace--Lytchak \cite[Proof of Corollary 1.5]{cl} we can even assume that we have a sequence in $X$ converging against $\xi$ such that the displacement function of $g$ on this sequence converges to $0$.

Recalling \cite[Definition II.10.62]{bridson_haefliger} we see that $C_g$ is a subgroup of the parabolic subgroup $G_\xi$ associated to $\xi$, which is defined as all the elements of $\GL(n;\IR)$ fixing $\xi$. Furthermore, $\langle g \rangle$ is a subgroup of the horospherical subgroup $N_\xi$ associated to $\xi$ (here we need that $\xi$ is the limit of a sequence in $X$ on which the displacement function of $g$ goes to $0$). This horospherical subgroup $N_\xi$ is a normal subgroup of $G_\xi$ and therefore we have a group homomorphism $C_g / \langle g \rangle \to G_\xi / N_\xi$ with kernel $(C_g \cap N_\xi) / \langle g \rangle$. Now we use the computations of $G_\xi$ and $N_\xi$ \cite[Propositions II.10.64 \& II.10.66]{bridson_haefliger} to conclude that the quotient $G_\xi / N_\xi$ is isomorphic to a product $\prod_{i=1}^k \GL(n_i;\IR)$, where the numbers $k$ and $n_i$ are determined by $\xi$. So the image of the map $C_g / \langle g \rangle \to G_\xi / N_\xi$ is a discrete subgroup of the Lie group $\prod_{i=1}^k \GL(n_i;\IR)$ and therefore has finite asymptotic dimension. So it remains to show that its kernel $(C_g \cap N_\xi) / \langle g \rangle$ has finite asymptotic dimension. But $N_\xi$ is a connected, unipotent Lie group. Since unipotent groups are nilpotent we conclude that $C_g \cap N_\xi$ is a discrete subgroup of a connected, nilpotent Lie group, and for such groups we know that the quotient $(C_g \cap N_\xi) / \langle g \rangle$ has finite asymptotic dimension.
\end{proof}

>From the discussion in this section we conclude that we can handle arithmetic groups, since they are discrete subgroups of almost connected real Lie groups. But note that the properties of arithmetic groups, that we will summarize in the next corollary, are all already known.

\begin{cor}
Arithmetic groups have finite asymptotic dimension, satisfy both Conjectures~\ref{conj:njdsf89} and \ref{conjdsf89023}, and lie in the class $E(\IQ)$ and hence satisfy the Burghelea conjecture.
\end{cor}

\subsection{\texorpdfstring{$3$}{3}-manifold groups}

\begin{thm}
Fundamental groups of compact $3$-manifolds all have finite asymptotic dimension, lie in the class $E(\IQ)$ and satisfy both Conjectures~\ref{conj:njdsf89} and \ref{conjdsf89023}.
\end{thm}

\begin{proof}
Let $M$ be a compact $3$-manifold. We can assume that $M$ is orientable. Indeed, if $M$ is not orientable, we take the orientable 
double cover of $M$ and reduce our discussion to the orientable case. By the fact that the class $E(\IQ)$ and the class of groups satisfying Conjectures~\ref{conj:njdsf89} and \ref{conjdsf89023} is closed under taking free products and taking subgroups,
we can make the following reductions:
\begin{itemize}
\item By doubling $M$ along the boundary we see that the group $\pi_1(M)$ 
is a subgroup of the fundamental group of a closed $3$-manifold (\cite[Lem.1.4.3]{f}). 
Thus we can assume that $M$ is closed. 
\item By prime decomposition theorem (\cite[Thm.1.2.1]{f}), every $3$-manifold admits a decomposition as a connected sum of prime manifolds, 
i.e., manifolds which cannot be further decomposed. 
A manifold which is prime and is not $S^2 \times S^1$ is irreducible. 
Thus we can assume that $M$ is irreducible.
\end{itemize}

First we have to check that rational homological dimension and asymptotic dimension of $\pi_1(M)$ is finite. 
Finiteness of $\hd_\IQ$ follows from Papakyriakopoulos' Sphere Theorem which implies that if $M$ is irreducible, 
then it is aspherical or $\pi_1(M)$ is finite. This also shows that the homological dimension of torsion-free fundamental groups of irreducible $3$-manifolds is finite. The finiteness of asymptotic dimension follows from the results of MacKay--Sisto \cite[Theorem 5.1]{mackay_sisto}.

Now we deal with reduced centralizers. By \cite[Theorem 2.5.1]{f} we have that for irreducible $3$-manifolds the centralizer $C_g$ of a nontrivial element $g$ is either abelian of rank at most $2$, or is isomorphic to a centralizer of some Seifert fibered submanifold of~$M$. By \cite[Theorem 1.8.1 \& Table 1]{f} we know the isomorphism types of fundamental groups of Seifert fibered manifolds and all of these types of groups are already treated in this paper or by other authors in earlier papers. Thus these centralizers lie in the class $E(\IQ)$ and satisfy Conjecture~\ref{conjdsf89023}. 
\end{proof}

\section{Counter-examples}\label{sec:counter.examples}

In this section we will construct two kinds of counter-examples to the Burghelea conjecture: 
the first one will be a finitely generated group having finite decomposition complexity and is based on an ad hoc modification of Burghelea's original construction of a counter-example combined with an embedding theorem guaranteeing embeddings of groups into finitely generated groups. 

Recall that a group $G$ is of type $F_\infty$, if there exists a simplicial model of its classifying space $BG$ with finite skeleta in every dimension. This especially implies that $G$ is finitely presentable.
Our second counter-example is of type $F_\infty$. The construction is based on Thompson's group $F$ and the fact that its homology and cohomology is explicitly known, and especially that we may find the (co-)homology of $\IC P^\infty$ inside it.

\subsection{Finite decomposition complexity}\label{sec:fdc}

\paragraph{Construction of the first example} Let us first review Burghelea's own counter-example \cite[Section~IV]{burghelea}. By the Kan--Thurston theorem, there is an aspherical space $X$ and a map $f \colon X \to \IC P^{\infty}$ which is a homology isomorphism, i.e., $f$ induces an isomorphism on all homology groups with arbitrary non-twisted coefficients. Let us now consider the following diagram:
\[\xymatrix{
Y \ar[d] \ar[r]^{\tilde{f}} & S^{\infty} \ar[d] \\
X \ar[r]^f & \IC P^{\infty}}\]

Here $S^{\infty} \to \IC P^{\infty}$ is the universal principal $S^1$-bundle and $Y$ is the pull-back of it along the map $f$. By the long exact sequence for the homotopy groups of a fibration, the space $Y$ is aspherical. Since $f$ is a homology isomorphism, it follows from the Serre spectral sequence that also $\tilde{f} \colon Y \to S^{\infty}$ is a homology equivalence, i.e., $Y$ is acyclic.

Let $G = \pi_1(Y)$ and $H = \pi_1(X)$. Note that the Kan--Thurston theorem can, in this case, produce a torsion-free group $H$, and therefore $G$ will be also torsion-free. Let $x \in G$ be a loop in one of the fibers of $Y$. Then $x$ is central in $G$ since $Y \to X$ is an oriented (even principal) $S^1$-bundle.

So we got a central extension
\[1 \to \IZ \to G \to H \to 1,\]
with $H_{2n}(H;R) \cong R$ and $H_{2n+1}(H;R) \cong 0$ for all $n \ge 0$ and any coefficient ring $R$. Since this central extension corresponds to the pull-back of the fibration $S^\infty \to \IC P^\infty$, we get that the corresponding Gysin homomorphism is an isomorphism, i.e., $T_0^G(x,R) \cong R$ for $x \in G$ being the image of $1 \in \IZ \hookrightarrow G$. So $G$ does not satisfy the Burghelea conjecture.

Summarizing, Burghelea \cite{burghelea} has proven the following result: there exists a central extension of groups
\[1 \to \IZ \to G \to H \to 1\]
such that
\begin{enumerate}
\item $H$ is homology equivalent to $\IC P^\infty$ and $G$ is a counter-example to the Burghelea conjecture, i.e., $T_0^G(x,R) \cong R$ for $x \in G$ being the image of $1 \in \IZ \hookrightarrow G$,
\item $G$ is acyclic, i.e., all its homology groups (in degrees $\ge 1$) with trivial coefficients vanish, and
\item both $G$ and $H$ are torsion-free.
\end{enumerate}

\paragraph{Finite decomposition complexity} Let us now show that we can carry out the above construction such that the group $H$ will additionally have finite decomposition complexity (Guentner--Tessera--Yu \cite{fdc_1}).

This will follow from a version of the Kan--Thurston theorem due to Leary \cite{leary.kan-thurston}: it is possible to construct the space $X$ such that it is a locally $\CAT(0)$ cubical complex, and such that it is the union of totally geodesic, locally $\CAT(0)$ sub-complexes $X^{(n)}$ which are homology equivalent to the sub-complexes $\IC P^n \subset \IC P^\infty$, are finite complexes and have dimension $2n$ (with the exception that $X^{(1)}$ has dimension $3$).

It follows that we have $H = \pi_1(X) \cong \colim \pi_1(X^{(n)})$, and Leary's construction also has the property that the structure maps in this colimit are injective. Since each complex $X^{(n)}$ is locally $\CAT(0)$ (and therefore its universal cover is contractible) and finite, we get that the groups $\pi_1(X^{(n)})$ are all torsion-free, and therefore also $H$. Note that $H$ is not necessarily of type $F_\infty$ since going from $X^{(n)}$ to $X^{(n+1)}$ attaches not only high-dimensional cubes in Leary's construction.

By a result of Wright \cite{wright_cube} we get that each $\pi_1(X^{(n)})$ has finite asymptotic dimension, and therefore $H$ itself has finite decomposition complexity (here we use that the structure maps of the colimit are injective, \cite{fdc}).

It follows that also the group $G$ is torsion-free and has finite decomposition complexity.

\paragraph{Finite generation}\label{par:fin.gen}

We will modify now the above contruction to get finitely generated counter-examples. Our modification will be an additional step, namely embedding the group $H$ into a finitely generated group $H_1$ while preserving its homological properties.

Let $F_2 = \langle a,b \rangle$ be the free group on two generators and consider a subgroup $F_{\infty} < F_2$ of infinite rank. Let $\{x_1,x_2,\ldots\}$ be a free basis of $F_{\infty}$. Now consider the group $H \asterisk F_2$ and let $\{h_1,h_2,\ldots\}$ enumerate all elements in $H$. Let $F'_{\infty} = \langle h_1 x_1,h_2 x_2,\ldots \rangle < H \asterisk F_2$. This group $F'_{\infty}$ is free and the indicated set of generators is a free basis of it. Let $\alpha \colon F_{\infty} \to F'_{\infty}$ be defined on the generators by $\alpha(x_i)=h_i x_i$. We define now $H_1$ to be the HNN extension of $H \asterisk F_2$ with respect to $\alpha$. The group $H_1$ is generated by $a,b,t$ and contains $H$ as a subgroup. It is also torsion-free (since it is an HNN extension of a torsion-free group) and still has finite decomposition complexity (since finite decomposition complexity is preserved by free products and HNN extensions).

Let us now show that the homology equivalence $f \colon BH \to \IC P^\infty$ extends over the map $BH \to BH_1$. Since $\IC P^\infty = K(\IZ,2)$, there is only one obstruction cocycle against such an extension. This cocycles lives in $H^3(H_1,H; \IZ)$. We will show further below that $H^3(H_1,H; \IZ) = 0$ and therefore the obstruction cocycle vanishes. Let us denote the extension of $f$ to $BH_1$ by $f_1 \colon BH_1 \to \IC P^\infty$. Using this map we can pull-back the principal $S^1$-bundle $S^\infty \to \IC P^\infty$ to $BH_1$. This produces as before a central extension of groups $1 \to \IZ \to G_1 \to H_1 \to 1$ and $G_1$ will be torsion-free, of finite decomposition complexity, and finitely generated.

The classifying space for $F_2$ can be chosen to be $1$-dimensional, and therefore, using the Mayer--Vietoris sequence for the cohomology of $H \asterisk F_2$ (see, e.g., \cite[Section VII.9]{brown}) we get $H^3(H;\IZ) \cong H^3(H \asterisk F_2;\IZ)$. Furthermore, using the long exact sequence for the cohomology of HNN extensions (see again \cite[Section VII.9]{brown}), and that the classifying space for $F_\infty$ can also be chosen to be $1$-dimensional, we get $H^3(H_1; \IZ) \cong H^3(H \asterisk F_2; \IZ)$. This shows $H^3(H_1,H; \IZ) = 0$ as claimed above.

The above arguments also show that $H_n(H;R) \cong H_n(H_1;R)$ for all $n \ge 3$ and all trivial coefficient rings $R$. Since the $S^1$-bundle over $BH$ is the pull-back of the $S^1$-bundle over $BH_1$, we conclude that $G_1$ is still a counter-example to the Burghelea conjecture. Furthermore, the Lyndon--Hochschild--Serre spectral sequence \eqref{eqjnk23fw} gives $H_n(G_1;R) = 0$ for all $n \ge 4$.

So we have proved that there is a central extension of groups
\begin{equation}
\label{eqjn22000r}
1 \to \IZ \to G_1 \to H_1 \to 1,
\end{equation}
such that:
\begin{enumerate}
\item Both groups $G_1$ and $H_1$ are finitely generated, torsion-free and have finite decomposition complexity.
\item $G_1$ is a counter-example to the Burghelea conjecture in the sense that $T_0^{G_1}(x,R) \cong R$ for $x \in G_1$ being the image of $1 \in \IZ \hookrightarrow G_1$.
\item We have $H_n(G_1;R) = 0$ for all $n \ge 4$ and coefficient rings $R$.
\end{enumerate}

\subsection{Thompson's group \texorpdfstring{$F$}{F}}\label{sec:F}

Brown \cite{brown_F} computed the homology and cohomology of Thompson's group $F$: there is a map $F \to [F,F] \times F_\ab$ which is both a homology and cohomology isomorphism.

We have $F_\ab \cong \IZ \times \IZ$ and the homology $H_\ast([F,F]; \IZ)$ admits a ring structure under which it is isomorphic to the polynomial ring $\IZ[t]$, where $t \in H_2([F,F]; \IZ)$. Furthermore, the cohomology ring $H^\ast([F,F]; \IZ)$ is isomorphic to the divided polynomial ring $\Gamma(u)$ with $\deg(u) = 2$. Recall that the latter is the subring of the polynomial ring $\IQ[u]$ generated additively by the elements $u^{(i)} = u^i / i!$ for $i \ge 0$. Since $u^{(i)} u^{(j)} = \binom{i+j}{i} u^{(i+j)}$ we see that the $\IZ$-span of the $u^{(i)}$, $i \ge 0$, is indeed a ring.

The cap-product with the element $u \otimes 1 \in H^2([F,F]; \IZ) \otimes H^0(F_\ab; \IZ) \hookrightarrow H^2(F; \IZ)$ acts on $H_\ast(F; \IZ) \cong H_\ast([F,F]; \IZ) \otimes H_\ast(F_\ab; \IZ)$ as $t^k \otimes y \mapsto k\cdot t^{k-1} \otimes y$, where $y \in H_\ast(F_\ab; \IZ)$ is arbitrary. So if now
\begin{equation}
\label{eqn2332}
1 \to \IZ \to G \to F \to 1
\end{equation}
is the central extension which is classified by $u \otimes 1 \in H^2(F; \IZ)$, then we see that $G$ is a counter-example to the Burghelea conjecture: using $\IQ$-coefficients the cap-product with $u \otimes 1$ is an isomorphism $H_\ast(F; \IQ) \to H_{\ast-2}(F; \IQ)$, and these groups are isomorphic to $\IQ \times \IQ$. Furthermore, the cap-product with $u \otimes 1$ is exactly the operator $S$ appearing in the definition of $T_\ast^{G}(x,\IQ)$, see \eqref{eqnjk3498fe}. Therefore $T_\ast^{G}(x,\IQ) \cong \IQ \times \IQ$ for $x \in G$ being the image of $1 \in \IZ \hookrightarrow G$ and for both $\ast=0,1$.

Note that this central extension corresponds to an oriented $S^1$-bundle $BG \to BF$. Since every $S^1$-bundle admits the structure of a principal $S^1$-bundle, for a proof of this see, e.g., Morita \cite[Proposition~6.15]{morita}, we conclude that there is a map $\phi \colon BF \to \IC P^\infty$ classifying this principal $S^1$-bundle.

It is known that $F$ is torsion-free and of type $F_\infty$, see Brown--Geoghegan \cite{brown_geoghegan}. 
This implies that $G$ is also of type $F_\infty$ (\cite[Proposition~2.7]{bieri} and \cite[Corollary~1.12]{bieri}).

To compute $H_*(G;\IQ)$ we use the Gysin exact sequence 
\begin{equation*}
\ldots \to H_{n+2}(G;\IQ) \to H_{n+2}(F;\IQ) \xrightarrow{-\cap (u \otimes 1)} H_n(F;\IQ) \to H_{n+1}(G;\IQ) \to \ldots
\end{equation*}
and conclude that $H_n(G;\IQ) = 0$ for $n \ge 3$.

Summarizing the above, we have constructed a central extension \eqref{eqn2332} with the following properties:
\begin{enumerate}
\item both groups $G$ and $F$ are torsion-free and of type $F_\infty$,
\item $G$ is a counter-example to the Burghelea conjecture since $T_\ast^{G}(x,\IQ) \cong \IQ \times \IQ$ for $x \in G$ being the image of $1 \in \IZ \hookrightarrow G$ and for both $\ast=0,1$, and
\item we have $H_n(G;\IQ) = 0$ for all $n \ge 3$.
\end{enumerate}

\bibliography{./burghelea}
\bibliographystyle{amsalpha}

\end{document}